\documentclass{article}
\usepackage{amsmath, amssymb, amsthm, relsize, tikz-cd, graphicx, hyperref}

\theoremstyle{plain}
\newtheorem{theorem}{Theorem}[section]
\newtheorem{lemma}[theorem]{Lemma}

\theoremstyle{definition}

\DeclareMathOperator*{\OmSum}{\mathlarger{\mathlarger{\Omega}}}

\newcommand{\tet}{\text{tet}_\beta}
\newcommand{\slog}{\text{slog}_\beta}

\begin{document}

\title{The Limits of a Family;\\\Large{of Asymptotic Solutions to The Tetration Equation}}

\author{James David Nixon\\
	JmsNxn92@gmail.com\\}

\maketitle

\begin{abstract}
In this paper we construct a family of holomorphic functions $\beta_\lambda (s)$ which are solutions to the asymptotic tetration equation. Each $\beta_\lambda$ satisfies the functional relationship ${\displaystyle \beta_\lambda(s+1) = \frac{e^{\beta_\lambda(s)}}{e^{-\lambda s} + 1}}$; which asymptotically converges as $\log \beta_\lambda(s+1) = \beta_\lambda (s) + \mathcal{O}(e^{-\lambda s})$ as $\Re(\lambda s) \to \infty$. This family of asymptotic solutions is used to construct a holomorphic function $\tet(s) : \mathbb{C}/(-\infty,-2] \to \mathbb{C}$ such that $\tet(s+1) = e^{\tet(s)}$ and $\tet : (-2,\infty) \to \mathbb{R}$ bijectively.
\end{abstract}



\section{Introduction}\label{sec1}
\setcounter{equation}{0}

This paper will start with a general theorem the author has shown a multitude of times, but of which most recently appears in \cite{Nix2,Nix3}. In \cite{Nix2} it is shown for a specific case, and modified to a real analysis scenario a couple more times; but in \cite{Nix3} the general theorem is given. We'll use this introduction to introduce the theorem, and talk a little bit about the notation.

\begin{theorem}\label{thmA}\footnote{We've added a proof of this theorem in the appendix \ref{app}.}
Let $\{H_j(s,z)\}_{j=1}^\infty$ be a sequence of holomorphic functions such that $H_j(s,z) : \mathcal{S} \times \mathcal{G} \to \mathcal{G}$ where $\mathcal{S}$ and $\mathcal{G}$ are domains in $\mathbb{C}$. Suppose there exists some $A \in \mathcal{G}$, such for all compact sets $\mathcal{N}\subset\mathcal{G}$, the following sum converges,

\[
\sum_{j=1}^\infty ||H_j(s,z) - A||_{z \in \mathcal{N},s \in \mathcal{S}} = \sum_{j=1}^\infty \sup_{z \in \mathcal{N},s \in \mathcal{S}}|H_j(s,z) - A| < \infty
\]

Then the expression,

\[
H(s) = \lim_{n\to\infty}\OmSum_{j=1}^n H_j(s,z)\bullet z = \lim_{n\to\infty} H_1(s,H_2(s,...H_n(s,z)))\\
\]

Converges uniformly for $s \in \mathcal{S}$ and $z \in \mathcal{N}$ as $n\to\infty$ to $H$, a holomorphic function in $s\in\mathcal{S}$, constant in $z$.
\end{theorem}

Upon which, this theorem provides a manner of proving holomorphy of an infinite composition of holomorphic functions; and it only requires that a certain sum converges. Where here, an infinite composition is denoted,

\[
\lim_{n\to\infty} H_1(s,H_2(s,...H_n(s,z))) = \OmSum_{j=1}^\infty H_j(s,z) \bullet z\\
\]

Much of the theory of infinite compositions, as the author has written about in \cite{Nix,Nix2,Nix3,Nix4} depends on the behaviour of a sum which we compare with the infinite composition. In the case of this paper (as in \cite{Nix2,Nix3}), the infinite composition falls into a degenerate category. This is the case that the value in $z$ of the infinite composition will be constant. And in contrast, in the non-degenerate category (as in \cite{Nix,Nix4}), we'd have that our infinite composition is holomorphic in $z$ and non-constant.

For this reason, we won't speak of $z$ at all, except to denote the manner of composition (like binding a variable to an integral, and then tossing it away afterwords). Instead, we'll be talking about two variables, $s,\lambda \in \mathbb{C}$. And discussing infinite compositions with these two variables.

This will birth us a two variable holomorphic function $\beta_\lambda(s)$; which we'll call a family of solutions to the asymptotic tetration equation. Where, for us, we'll call a function $l(s)$ a solution to the asymptotic tetration equation if,

\[
\log l(s+1) - l(s) \to 0\,\,\text{as}\,\,|s|\to\infty\\
\]

Where we'll mostly be concerned with $|s|\to\infty$ while $s$ is in a half-plane; and these things are holomorphic (or with countable singularities) unless stated otherwise. These asymptotic solutions, essentially look like tetration at infinity, but everywhere else they may not look like tetration. This allows us to talk about logarithms of these things at infinity in a nice manner. And if we are able to solve the equation $\log F(s+1) = F(s)$ for large $s$, repeatedly taking logarithms allows us to extend this definition almost everywhere in $\mathbb{C}$. 

So the idea is to take our family of asymptotic solutions $\beta_\lambda$ and construct an error term $\tau_\lambda$ which solves the tetration equation for large $s$. Since $\beta_\lambda$ will be a well behaved solution to the asymptotic tetration equation; this is doable.

\section{The family of functions $\beta_\lambda$}\label{sec2}
\setcounter{equation}{0}

We will start our foray by pulling out of a hat the sequence of functions we want to infinitely compose to get $\beta_\lambda$. We'll denote this sequence of functions $q_j(s,\lambda,z)$; where the $z$ value will disappear in the end. Write,

\[
q_j(s,\lambda,z) = \frac{e^z}{e^{\lambda (j-s)} + 1}\\
\]

Where the index $j \in \mathbb{N}$ and $j\ge 1$. Now, we're going to force $\Re(\lambda) > 0$ and that $\lambda(j-s) \neq (2k+1)\pi i$ for all $k\in\mathbb{Z}$. The first restriction will be needed for the summation, and the second restriction ensures we have no poles. We'll call this domain of holomorphy $\mathbb{L}$, in which each $q_j(s,\lambda,z) : \mathbb{L} \times \mathbb{C} \to \mathbb{C}$, where $(s,\lambda) \in \mathbb{L}$ and $z \in \mathbb{C}$. 

Observe that,

\[
\sum_{j=1}^\infty |q_j(s,\lambda,z) | = \sum_{j=1}^\infty |\frac{e^z}{e^{\lambda (j-s)} + 1}| < \infty\\
\]

But, even better than this, we have a normally converging sum. Let $\mathcal{K} \subset \mathbb{C}$ and let $\mathcal{U} \subset \mathbb{L}$ both be compact sets. Then,

\[
\sum_{j=1}^\infty ||q_j(s,\lambda,z) ||_{\mathcal{U},\mathcal{K}} < \infty\\
\]

This should tell us that Theorem \ref{thmA} is going to be useful, as $q_j$ satisfies all the properties of $H_j$ in the theorem's statement. Now, a small reminder is that Theorem \ref{thmA} has no restriction on how many variables are involved, although it's only stated for one variable. For clarification of this, the reader is pointed to \cite{Nix3}; or to the proof of Theorem \ref{thmA} attached in the appendix. Therein, if we take,

\[
\beta_\lambda(s) = \OmSum_{j=1}^\infty q_j(s,\lambda,z)\bullet z\\
\]

Then $\beta_\lambda(s)$ is holomorphic for $(s,\lambda) \in \mathbb{L}$. It's important to remember what $\beta_\lambda$ looks like though, and why we'd even want this function. We write,

\[
\beta_\lambda(s) = \OmSum_{j=1}^\infty \frac{e^z}{e^{\lambda(j-s)} +1}\bullet z\\
\]

Then, if we shift the argument in $s$ forward by $1$, we get something magical.

\begin{eqnarray*}
\beta_\lambda(s+1) &=& \OmSum_{j=1}^\infty \frac{e^z}{e^{\lambda(j-s-1)} +1}\bullet z\\
&=& \OmSum_{j=0}^\infty \frac{e^z}{e^{\lambda(j-s)} +1}\bullet z\,\,\text{We've shifted the index here}\\
&=& \frac{e^{\OmSum_{j=1}^\infty \frac{e^z}{e^{\lambda(j-s)} +1}\bullet z}}{e^{-\lambda s} + 1}\\
&=& \frac{e^{\beta_\lambda(s)}}{e^{-\lambda s} + 1}\\
\end{eqnarray*}

And from this, we're in a position to state that this is a family of solutions to the asymptotic tetration equation. That is,

\[
\log \beta_\lambda(s+1) - \beta_\lambda (s) = -\log(1+e^{-\lambda s})\\
\]

Which tends to $0$ exponentially as $|s| \to \infty$ while $\Re(\lambda s ) > 0$. Now this form of the $\beta_\lambda$ family is difficult to compute, we need to compute a bunch of nested exponentials--at infinity no less, and so therefore it can be easier to make a change of variables $s = \log(w)/\lambda$. The author would like to thank Sheldon Levenstein for doing this first; as the author rarely numerically evaluates, it was Sheldon's observation that this form is much less exhausting computationally.

Write,

\[
g_\lambda(w) = \beta_\lambda (s)\\
\]

Then this is holomorphic when $w \neq -e^{\lambda j}$,

\[
g_\lambda(w) = \OmSum_{j=1}^\infty \frac{we^z}{e^{\lambda j} + w}\bullet z\\
\]

Upon which, calculating Taylor coefficients for $g_\lambda(w)$ at $0$ are surprisingly simple; especially by the functional equation,

\[
g_{\lambda} (e^{\lambda} w) = \frac{w}{w+1} e^{g_\lambda(w)}\\
\]

Where now computing the Taylor coefficients at $0$ is a relatively simple procedure; it's inductive. This will construct a Taylor-series valid for $|w| < e^{\Re\lambda}$. And to extend $g_\lambda$ to its maximal domain we just iterate the functional equation. We sketch the process for the curious reader who wants to numerically evaluate these functions. 

Call $g^{(k)}_\lambda(0) = a_k$ and $b_k = \frac{d^k}{dw^k}\Big{|}_{w=0} e^{g_\lambda(w)}$; then this is given as,

\begin{eqnarray*}
e^{\lambda k}g^{(k)}_\lambda(0) &=& \sum_{c=0}^k\dbinom{k}{c} \left(\frac{d^{k-c}}{dw^{k-c}}\Big{|}_{w=0} \frac{w}{w+1}\right) \left(\frac{d^c}{dw^c}\Big{|}_{w=0} e^{g_\lambda(w)}\right)\\
&=&\sum_{c=0}^{k-1} \dbinom{k}{c} (k-c)!(-1)^{k-c+1} \frac{d^c}{dw^c}\big{|}_{w=0} e^{g_\lambda(w)}\\
e^{\lambda k}a_k &=& k!(-1)^{k+1}\sum_{c=0}^{k-1}\frac{(-1)^{c}}{c!} b_c\\
b_{c} &=& \sum_{d=0}^{c-1} \dbinom{c-1}{d} b_d a_{c-d}\\
\end{eqnarray*}

Where,

\[
g_\lambda(w) = \sum_{k=0}^\infty a_k \frac{w^k}{k!}\,\,\text{for}\,\,|w| < e^{\Re \lambda}\\
\]

Which is the process which generates our coefficients. Below we've attached some graphs of these functions. When $\lambda = \log(2)$; and when we're $n$ compositions deep, $g_\lambda(x)$ looks like Figure \ref{figg1}.

\begin{figure*}[htp]
  \centering
  \includegraphics[scale=0.5]{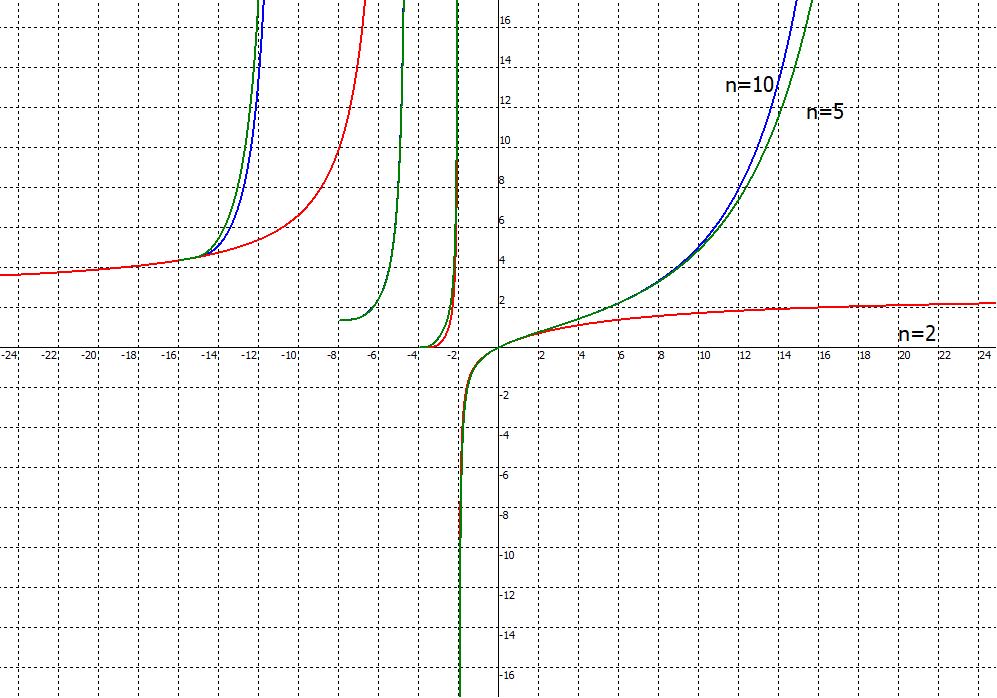}\quad
  \hfill
  \caption {The function $g_{\log(2)}(x)$ for $n=2,5,10$ iterations.}
  \label{figg1}
\end{figure*}

Convergence is very fast in the infinite composition manner. And we can begin to see the rapid growth. This will eventually start to grow faster than exponentiation by the law $g(2x) = \frac{x}{x+1}e^{g(x)}$; but it starts out balanced and well behaved. We can also clearly see the essential singularities beginning to form at $x = -2^j$ for $j \ge 2$ and the pole at $x=-2$. In Figure \ref{figLOG}, we can see this functions behaviour in the complex plane.\\

\begin{figure}
  \centering
  \includegraphics[scale=0.7]{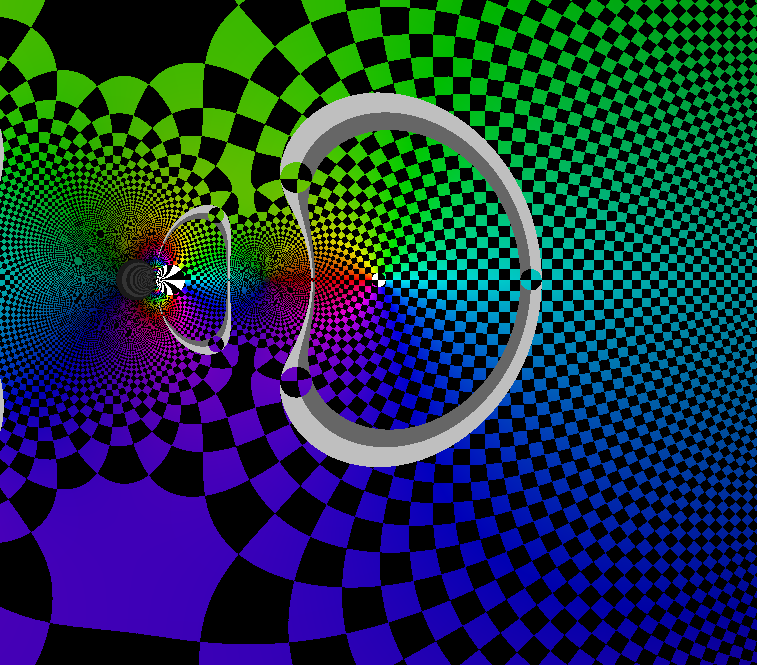}\quad
  \hfill
  \caption {The function $g_{\log(2)}(w)$ for $n=10$ iterations.}
  \label{figLOG}
\end{figure}

For the value $\lambda = 1/2 + 3i$ we've also included a hue plot of $g_\lambda(w)$ in Figure \ref{figGALT}. This is solely done for the tenth iteration--and is good enough for local values. This graph includes grid-lines, the unit disk, and a black and white checkered marker for the origin.\\

\begin{figure*}
  \centering
  \includegraphics[scale=0.6]{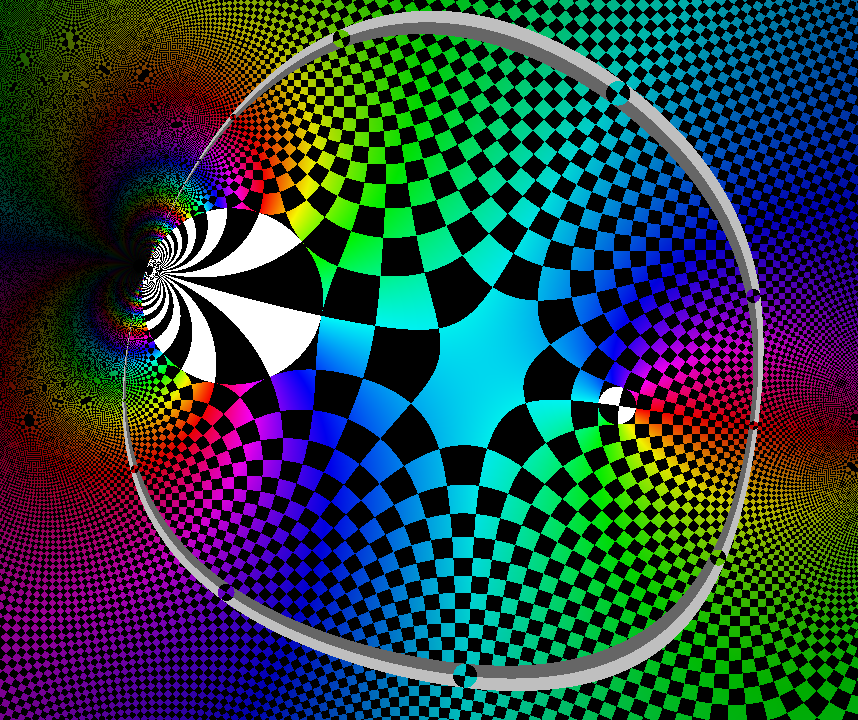}\quad
  \hfill
  \caption {The function $g_{\lambda}(w)$ about $w=0$ for $n=10$ iterations; $\lambda = 1/2 + 3i$.}
  \label{figGALT}
\end{figure*}

In Figure \ref{figGALT2} we can see this graph zoomed out further; where we're bound to have a numerical discrepancy because we've only used $n=10$ iterations.\\

\begin{figure*}[htp]
  \centering
  \includegraphics[scale=0.6]{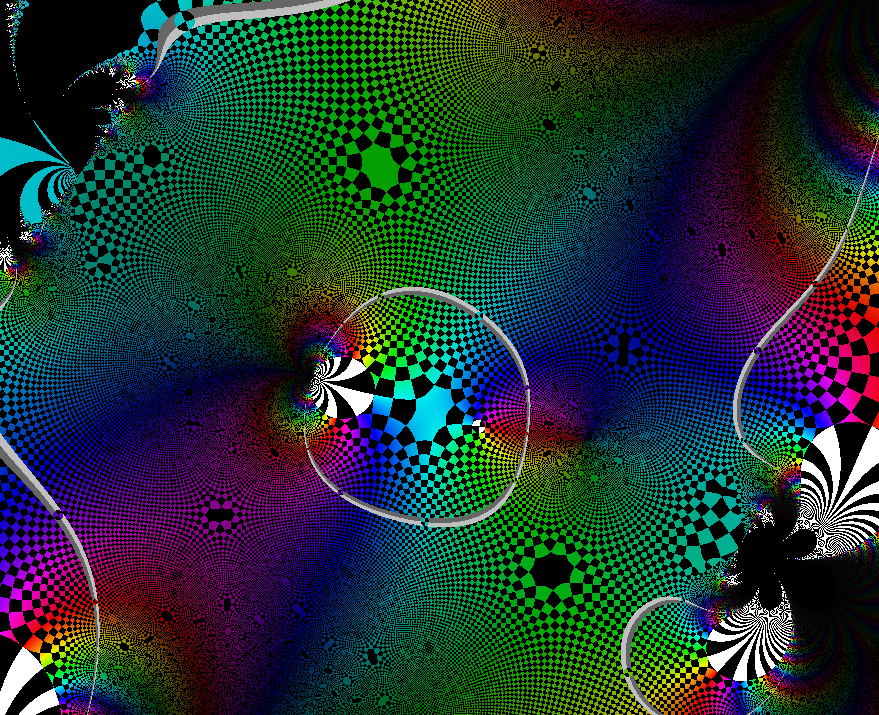}\quad
  \hfill
  \caption {The function $g_{\lambda}(w)$ about $w=0$ for $n=10$ iterations; $\lambda = 1/2 + 3i$; zoomed out further.}
  \label{figGALT2}
\end{figure*}

As you can see, there is a lot of chaos in these functions. They display essential singularity behaviour; and don't behave as one would expect a holomorphic function to behave as $w \to \infty$. It's our job to regulate these functions.\\

Sheldon Levenstein, is again, to thank for this change of variables. This form of many of the equations the author has solved, make the solutions look like a kind of mock Schr\"{o}der equation. Where in Schr\"{o}der's case one would solve,

\[
\Psi(Ls) = e^{\Psi(s)}\\
\]

We are solving something similar, but adding a multiplicative factor to the construction. This helps tremendously at manipulating the complex dynamics of these objects. And the convergents can make a very complicated thing less so. Plus, computationally it compares to calculating $w$ versus calculating $e^s$--and so avoids overflow errors that much better.\\ 

It's important to also note that $\beta_\lambda(s)$ has an exponential series--from $g_\lambda(w)$'s Taylor series. Which is,

\[
\beta_\lambda(s) = \sum_{k=1}^\infty a_k\frac{e^{k\lambda s}}{k!}\\
\]

Which is valid for $\Re(s) < 1$. Which implies that $\beta_\lambda(s+\frac{2\pi i}{\lambda}) = \beta_\lambda(s)$, so that our function is periodic in the $s$ argument. You can also see this by inspection, plugging in the value in the infinite composition.

We compress all this knowledge into the existence of a family of functions which solve the asymptotic tetration equation.

\begin{theorem}[Family Of Asymptotic Tetration Functions]\label{thmFAM}
There exists a family of functions $\beta_\lambda$ which are holomorphic on $\mathbb{L} =\{(s,\lambda) \in \mathbb{C}^2\,|\, \Re(\lambda) > 0,\,\,\lambda(j-s) \neq (2k+1)\pi i,\,\,j,k \in \mathbb{Z},\,\,j\ge1\}$. These functions are expressible as,

\[
\beta_\lambda(s) = \OmSum_{j=1}^\infty \frac{e^z}{e^{\lambda(j-s)} + 1}\bullet z\\
\]

Satisfy the functional equation,

\[
\beta_\lambda(s+1) = \frac{e^{\beta_\lambda(s)}}{e^{-\lambda s} + 1}\\
\]

And the asymptotic relationship,

\[
\log(\beta_\lambda(s+1)) - \beta_\lambda(s) = \mathcal{O}(e^{-\lambda s})\\
\]

As $|s| \to \infty$, wherever $\Re(\lambda s) > 0$.
\end{theorem}

\begin{proof}
For convergence: see Theorem \ref{thmA}--see Appendix (or \cite{Nix3}) for proof. The functional equation is given from convergence. The asymptotics, by the functional equation.
\end{proof}

In Figure \ref{figBETALOG} is attached a graph of $\beta_{log(2)}(x)$ for $x \in [-10,4]$ for $n=100$ iterations. We can clearly see the beginning of our super-exponential growth. Trying to go further out in the $x$ variable will cause overflow errors very fast.

\begin{figure}
    \centering
    \includegraphics[scale=0.5]{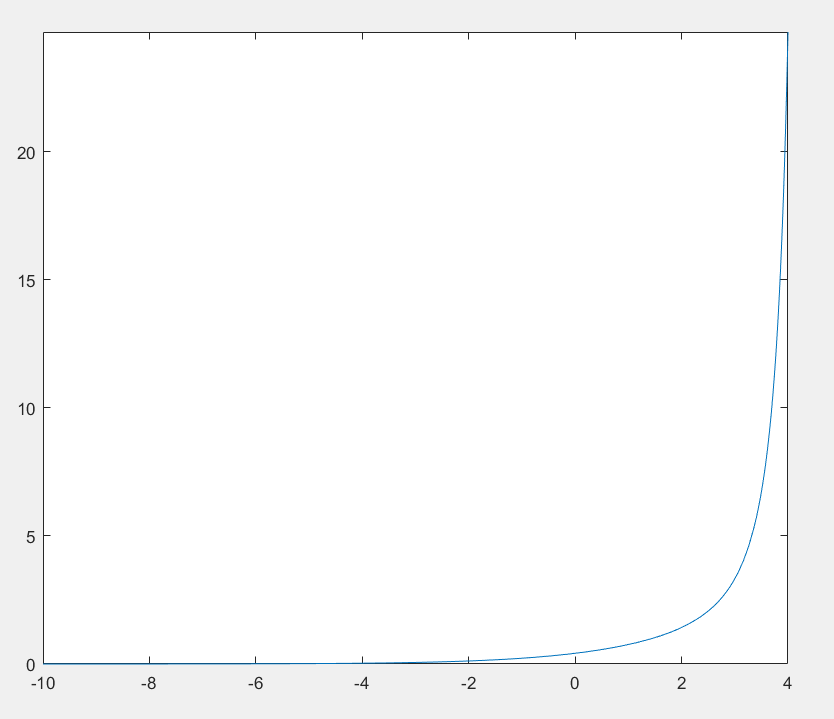}
    \caption{The function $\beta_{log(2)}(x)$ at $n=100$ iterations.}
    \label{figBETALOG}
\end{figure}

\section{The exponential convergents}\label{sec3}
\setcounter{equation}{0}

In this section we'll focus on better approximating tetration using $\beta_\lambda(s)$ at infinity. This is a difficult idea to intuit, but we're going to better understand its behaviour at infinity. The first thing we'd like to do is construct a sequence of convergents. Let's call,

\begin{eqnarray*}
\log(\beta_\lambda(s+1)) - \beta_\lambda(s) &=& \tau_\lambda^1(s) = -\log(1+e^{-\lambda s})\\
\log \log (\beta_\lambda(s+2)) - \beta_\lambda(s) &=& \tau_\lambda^2(s)\\
\log \log \log (\beta_\lambda(s+3)) - \beta_\lambda(s) &=& \tau_\lambda^3(s)\\
&\vdots&\\
\log^{\circ n} \beta_\lambda (s+n) - \beta_\lambda (s)&=& \tau_\lambda^n(s)\\
\end{eqnarray*}

Upon which, the asymptotic relationship,

\[
\log(1+\frac{\tau_\lambda^n(s+1)}{\beta_\lambda(s+1)}) - \tau_\lambda^{n+1}(s) = \log(1+e^{-\lambda s}) = \mathcal{O}(e^{-\lambda s})\\
\]

Which is another side effect of being an asymptotic solution to tetration; but it's required we have exponential convergence of $\tau$ as $\Re(s) \to \infty$. To make sure everything stays well behaved in the iterated logarithm.

To accomplish this, we need that $\beta_\lambda(s) \to \infty$ as $\Re(s) \to \infty$. This paper heavily depends on this result--and the proof is something rather unique to the exponential $e^z$. This is our most heavily cited result from Milnor--alors, one needs to understand it to understand this result.

\begin{theorem}[The Unbounded Theorem]\label{thmUNB}
The function $\beta_\lambda(s) \to \infty$ as $\Re(s) \to \infty$.
\end{theorem}

\begin{proof}
The function $\beta_\lambda(s)$ looks like the orbits of $e^z$ as $\Re(s) \to \infty$. The accumulation points of $e^z$ for $z$ almost everywhere (under the Lebesgue measure) are the orbit $1,e,e^e,e^{e^e},...$; the orbit of the exponential at $0$; which diverge to infinity. This fact is cited from Milnor \cite{Mil}; but was actually shown by Lyubich \cite{Lyu} and Rees \cite{Ree}. Milnor even makes a joke that overflow errors eventually equal $0$ after a small amount of iterations

For our case, since $\beta(s+1) = e^{\beta(s)}(1+e^{-\lambda s})^{-1} = e^{\beta(s) - \log(1+e^{-\lambda s})} = e^{\beta(s) - \epsilon}$; for very large $\Re(s)$ we eventually converge towards the orbit $1, e,e^e,e^{e^e},...,$. As these orbits are in $\epsilon$-neighborhoods of each other at infinity--with $\epsilon \to 0$.\\

Take a neighborhood $\mathcal{N}_\epsilon^s = \{z \in \mathbb{C}\,|\, |z-\beta(s)| < 2\epsilon)\}$; then;

\[
\beta(s+1) \in \exp(\mathcal{N}_\epsilon^s) \supset \mathcal{N}_{\epsilon_1}^{s+1}\\
\]

For an appropriate $\epsilon_1$. Iterating this procedure looks like the iterated orbits $\exp^{\circ n}(\mathcal{N}_\epsilon^s) \supset \mathcal{N}_{\epsilon_n}^{s+n}$. These must be dense in $\mathbb{C}$. But additionally; each starting point $z \in \mathcal{N}_\epsilon^{s}$ (almost everywhere) diverges to infinity as $n\to\infty$ under the application $\exp^{\circ n}(z)$. To see this, each $z_n = \exp^{\circ n}(z)$ gets arbitrarily close to $1$--for almost all $z$. Then, for $|z^*-z| < \delta$, we must have $\exp^{\circ n}(z^*) = 1$. The values $z^*$ cover the neighbourhood $\mathcal{N}_\epsilon^s$ as $\delta$ can be found as small as possible by letting $n$ be large enough.

Each point $\beta_\lambda(s+n) \in \mathcal{N}^{s+n}_{\epsilon_n} \subset \exp^{\circ n}(\mathcal{N}_\epsilon^{s})$. Therefore $\lim_{n\to\infty} \beta_\lambda(s +n) \to \infty$ for almost all $s \in \mathbb{C}$--excluding a measure zero set under the Lebesgue measure. The points of non-divergence are points $z$ where $\exp^{\circ n}(z)$ forms a cycle (or a fixed point). No cycle can occur for $\beta_\lambda(s+n)$--by its functional equation. And since all cycles and fixed points are repelling of $e^z$; any neighborhood of a cycle diverges. Thus, the equation $\lim_{\Re(s) \to \infty} \beta_\lambda(s) = \infty$ is valid for all $s$.
\end{proof}

We are going to call upon the sequence of functions $\tau_\lambda^n(s)$, which are holomorphic on some subset $\mathbb{L}$ (we'll get to that later), and try to express them.

To begin, we know that each $\tau_\lambda^n(s)$ has exponential decay to $0$ as $\Re(s) \to \infty$. To show this, we need only look at the functional equation. Assume it follows for $n$ and go by induction to get $n+1$. Take,

\[
\tau_\lambda^{n+1}(s) = -\log(1+e^{-\lambda s}) + \log(1+\frac{\tau_\lambda^{n}(s+1)}{\beta_\lambda(s+1)})\\
\]

Since we know that, by The Unbounded Theorem \ref{thmUNB},

\[
\frac{1}{\beta_\lambda(s+1)} \to 0\,\,\text{as}\,\,\Re(s) \to \infty\\
\]

We know that $\tau_\lambda^{n+1}(s)$ must look like $-\log(1+e^{-\lambda s}) + \mathcal{O}(e^{-\lambda s})$. Which certainly means each $\tau$ has exponential convergence to $0$. But we want something slightly stronger. To continue, 

\[
\tau_\lambda^n(s) = \log^{\circ n} \beta_\lambda (s+n) - \beta_\lambda (s)\\
\]

And we want to show by induction that this thing decays like $-e^{-\lambda s}(1 + o(1))$. These functions satisfy the identity,

\[
\tau_\lambda^{n+1}(s) = \log(\beta_\lambda(s+1) + \tau_{\lambda}^n(s+1)) - \beta_\lambda(s)\\
\]

We have begun this iteration with $\tau_\lambda^0(s) = 0$ and,

\[
\tau_\lambda^1(s) = \sum_{k=1}^\infty \frac{(-1)^k}{k} e^{-k\lambda s}\\
\]

Now the goal is, when talking about $\tau_\lambda^2(s)$, we can almost express it as an exponential series; and similarly with $\tau_\lambda^n$. At least, in a neighborhood of $\infty$. This lets us say that $\tau_\lambda^n(s)$ has a removable singularity at $\Re(s) = \infty$--which will suffice to show the limit $n\to\infty$ converges. We skip straight to the theorem.

\begin{theorem}[The Removable Singularity Theorem]\label{thmRMV}
The functions $\tau_\lambda^n(s)$ are holomorphic for $(s,\lambda) \in \mathbb{L}$ for large enough $\Re(s) > T$ depending on $|\Im(s)| \le T'$; and satisfy,

\[
\tau_\lambda^n(s) = -\log\left(1+e^{-\lambda s}\right) + o(e^{-\lambda s})\\
\]

As $\Re(s) \to \infty$. And the $o(e^{-\lambda s})$ term looks like $\sum_{j=1}^{n-1} \frac{q^je^{-\lambda s}}{P_j(s)}$ for some $0 < e^{-\Re \lambda} < q < 1$. Where,

\[
P_j(s) = \prod_{k=1}^j \beta_\lambda(s+k)\\
\]
\end{theorem}

\begin{proof}
We go by induction on $n$. The initial condition is trivally true because, when $n=1$, the function $\tau_\lambda^1(s) = -\log(1+e^{-\lambda s})$. Assume the result for $n$. Then,

\[
\tau_{\lambda}^{n+1}(s) = -\log(1+e^{-\lambda s}) + \log\left(1+ \frac{\tau_\lambda^n(s+1)}{\beta_\lambda(s+1)}\right)\\
\]

This function is holomorphic for $\lambda (j-s) \neq (2k+1)\pi i$, excluding where; $\beta(s+1) = - \tau_\lambda^n(s+1)$. But $\beta(s+1) \to \infty$ and $\tau_\lambda^n(s) = \mathcal{O}(\sum_{j=1}^{n-1}q^je^{-\lambda s})$. This confirms holomorphy for $\Re(s) > T$ and $|\Im(s)| \le T'$. So, we can expand the second term in a Taylor series;

\[
\log\left (1+ \frac{\tau_\lambda^n(s+1)}{\beta_\lambda(s+1)} \right ) = \sum_{k=1}^\infty \frac{(-1)^{k-1}}{k} \left( \frac{\tau_\lambda^n(s+1)}{\beta_\lambda(s+1)}\right)^k\\
\]

Therefore, using the bound, $\log(1+w) \le (1+\epsilon)|w|$ for $\epsilon \to 0$ as $w \to 0$; writing $(1+\epsilon)e^{-\Re\lambda} < q < 1$:

\begin{eqnarray*}
\tau_\lambda^{n+1}(s) + \log(1+e^{-\lambda s}) &=& \log\left(1+ \frac{\tau_\lambda^n(s+1)}{\beta_\lambda(s+1)}\right)\\
&=&\log\left(1+ \frac{-\log(1+e^{-\lambda (s+1)}) + o(\sum_{j=1}^{n-1} q^je^{-\lambda (s+1)})}{\beta_\lambda(s+1)}\right)\\
&=&(1+\epsilon)\frac{-\log(1+e^{-\lambda (s+1)}) }{\beta_\lambda(s+1)} + o(\sum_{j=1}^{n-1} q^{j+1}e^{-\lambda s})/\beta_\lambda(s+1)\\
\end{eqnarray*}

And, since $\beta_\lambda(s) \to \infty$ as $\Re(s) \to \infty$,

\[
\frac{-\log(1+e^{-\lambda (s+1)}) }{\beta_\lambda(s+1)e^{-\lambda s}} \to \frac{-e^{-\lambda}}{\beta_\lambda(s+1)} \to 0\\
\]

Since $(1+\epsilon)e^{-\Re \lambda} < q < 1$, we can bound the first term $q \frac{e^{-\lambda s}}{\beta_\lambda(s+1)}$. Therefore,

\[
\frac{\tau_\lambda^{n+1}(s) + \log(1+e^{-\lambda s})}{e^{-\lambda s}} \to 0\\
\]

The term $P_j(s+1) \cdot \beta_\lambda(s+1) = P_{j+1}(s)$--and so the recursion on the $o$ term is satisfied. Thus the limit is satisfied as $\Re(s) \to \infty$; and the $o$-term decays at least like $\sum_{j=1}^n\frac{q^je^{-\lambda s}}{P_j(s)}$ for some $e^{-\Re \lambda} < q < 1$.
\end{proof}

The purpose of this $o$ term is pretty straight forward. It allows us to bound the function $\tau_\lambda^n(s)$ independent of the index $n$. This means, we have a normality condition on the family of functions $\{\tau_\lambda^n(s)\}_{n=1}^\infty$. If we call the function,

\[
P_\lambda(s,q) = \sum_{j=1}^\infty \frac{q^j}{|P_j(s)|} = \sum_{j=1}^\infty \dfrac{q^j}{\prod_{k=1}^j |\beta_\lambda(s+k)|}\\
\]

Then our sequence of functions,

\[
||\tau_\lambda^n(s) + \log(1+e^{-\lambda s})||_{\mathcal{K}} \le D||e^{-\lambda s}P_\lambda(s,q)||_{\mathcal{K}}\\
\]

For some $0 < e^{-\Re \lambda} < q < 1$ and $D>1$ dependent on a compact set $\mathcal{K} \subset \mathbb{L}$--so long as the real argument of $s$ in $\mathcal{K}$ is sufficiently large. This bound is created on $\lambda$ and $s$; and is uniform in both variables. Where as we increase the real argument of the set $\mathcal{K}$ we get $D \to 1$ and $q \to e^{-\Re \lambda}$. We write a quick theorem, and give a quick proof.

\begin{theorem}[Normality At Infinity Theorem]\label{thmNRMINF}
For a compact set $\mathcal{K} \subset \mathbb{L}\cup \{\infty\}$; in which $\Re(s) > K$, and $\infty$ is interpreted as $\Re(s) \to \infty$; the sequence of functions $\tau_\lambda^n(s)$ are normal on $\mathcal{K}$ (including the point at infinity)--and satisfy the bound:

\[
||\tau_\lambda^n(s) + \log(1+e^{-\lambda s})||_{\mathcal{K}} \le D||e^{-\lambda s}P_\lambda(s,q)||_{\mathcal{K}}\\
\]

For $D> 1$ and $0 < e^{-\Re \lambda} < q < 1$. Where $D \to 1$ and $q \to e^{-\Re \lambda}$ as the minimum real argument of $\mathcal{K}$ grows.
\end{theorem}

\begin{proof}
Each $\tau_\lambda^n(s)$ looks like,

\[
\tau_\lambda^n(s) = -\log(1+e^{-\lambda s}) + \mathcal{O}(\sum_{j=1}^{n-1} \frac{q^je^{-\lambda s}}{P_j(s)})\\
\]

For a value $0 < e^{-\Re \lambda} < q < 1$; with $q \to e^{-\Re \lambda}$ as $\Re(s) \to \infty$. The series,

\[
P_\lambda(s,q) = \sum_{j=1}^\infty \frac{q^j}{|P_j(s)|} = \sum_{j=1}^\infty \dfrac{q^j}{\prod_{k=1}^j|\beta_\lambda(s+k)|} < \infty
\]

And converges uniformly for $(s,\lambda) \in \mathcal{K} \subset \mathbb{L}$; including the point at infinity where $P_\lambda(\infty) = 0$. We can add these terms up to $\infty$ and we must have,

\[
\tau_\lambda^n(s) = -\log(1+e^{-\lambda s}) + \mathcal{O}(e^{-\lambda s}P_\lambda(s,q))\\
\]

And the constant on this $\mathcal{O}$ term is independent of $n$ by The Removable Singularity Theorem \ref{thmRMV}, and the $o$ term we derived. Thus, we can write,

\[
||\tau_\lambda^n(s) + \log(1+e^{-\lambda s})||_{\mathcal{K}} \le D||e^{-\lambda s}P_\lambda(s,q)||_{\mathcal{K}}\\
\]

For some $D>1$. We must have $D \to 1$ and $q \to e^{-\Re \lambda}$ as the minimum real argument of $\mathcal{K}$ grows; because the $o$-term for each $n$ looks like $\sum_{j=1}^{n-1} \dfrac{e^{-\lambda (s+j)}}{\prod_{k=1}^j \beta_\lambda (s+k)}$ as $\Re(s) \to \infty$.
\end{proof}

\section{Choosing the proper Riemann mapping}\label{sec4}
\setcounter{equation}{0}

The author would like to take the reader through a brief segue in the history of tetration. The year is 1950, and Kneser has published a treatise on iterating the exponential function: \textit{Reelle analytische Losungen der Gleichung $\varphi (\varphi (x))=e^{x}$ und verwandter Funktionalgleichungen} \cite{Kne}. In constructing the iterate of the exponential he had a god knows how revelation. Before the dawn of computers, before the dawn of an efficient way at calculating these things, Kneser saw a solution to tetration.

Now the idea isn't so far out there now. But this idea is still what us ``tetrationers'' think of when we think of a nice solution to tetration. Hell, as far as most of us are concerned, it's the nicest tetration. Hell, at this point in history, it's still \emph{the} tetration.

And what was Kneser's \emph{je ne sais quoi} that flipped everything on its head?.. A Riemann mapping. Nothing more, nothing less. It was the key to his tetration; a Riemann mapping.

The idea was simple enough; we take the inverse Schr\"{o}der function $\Psi$ of $e^z$ about a fixed point $L \in \mathbb{C}$. This function $\Psi$ is entire, and satisfies,

\[
e^{\Psi(z)} = \Psi(Lz)\\
\]

Where, the standard way to iterate this is to write,

\[
F(z) = \Psi(e^{L(z-z_0)})\\
\]

Upon which,

\[
F(z+1) = e^{F(z)}\\
\]

Where, this may betray the simplicity, it looks right, it should be our solution. The problem is, by construction, this function $F: \mathbb{R} \not\to \mathbb{R}$. This tetration is not real-valued. We definitely want the iteration of exponentiation to be real-valued. So, this form is pretty useless.

Now, Kneser's idea was simple, but breathtaking. Instead of talking just about a fixed point $L$, we also talk about its conjugate pair $\overline{L}$. And we want to create a Riemann mapping which glues the two tetrations from above into a single entity. And this will be real-valued, so long as we remember to keep them conjugate similar.

It's difficult to find analyses of Kneser's method, as his paper is in German, and it's yet to be translated. However many people have re-explained his work. The most acute, I find, is in Sheldon Levenstein's interpretation. He simplifies looking for the Riemann mapping by framing it as a search for a periodic function. In many ways, he refers to Kneser's construction in a much more hands-on, crunch the numbers, approach.

To explain, if we call $\mathbb{H} = \{z \in \mathbb{C}\,|\,\Im(z) > 0\}$, there is a $1$-periodic function $\theta(z) : \mathbb{H} \to \mathbb{C}$ such that,

\[
\theta(z) = \sum_{k=0}^\infty p_ke^{2\pi i k z}\\
\]

\[
\text{tet}_K(z) = \Psi(e^{Lz}\theta(z))
\]

Where,

\[
\overline{\text{tet}_K(z)} = \overline{\Psi(e^{Lz}\theta(z))}\\
\]

Is precisely,

\[
\text{tet}_K(\overline{z})\\
\]

Where the conjugate of $\Psi$ is the inverse Schr\"{o}der function about the fixed point $\overline{L}$, and the conjugate $\theta$ function works in the same manner here.

Now, finding this function $\theta$ is very difficult, and no matter how you slice it, you need to compute a Riemann mapping. That Riemann mapping being a miraculous thing. Where firstly, Kneser finds a nice simply connected domain $\mathbb{E}$--where on the boundary $\Psi$ is real-valued, and then produces the Riemann mapping to $\mathbb{H}$, all while respecting the Abel functional equation, and all while making sure this will be real-valued.

And so, at this point in our paper, we have to pay respect that in constructing a holomorphic tetration--Kneser had to pull out a very complicated mapping theorem. And so, any solution to tetration avoiding some kind of complexity similar to this, would either be miraculous or just wrong.\\

In our approximate solutions to tetration, in the asymptotic, they have singularities. They have singularities everywhere. When we talk about,

\[
F_\lambda^n(s) = \beta_\lambda(s) + \tau_\lambda^n(s)\\
\]

We have to account for a plethora of singularities which occur when $\lambda(j-s) = (2 k+1)\pi i$. These are very ugly singularities too, where if,

\begin{eqnarray*}
\beta_\lambda(s) &=& \infty\\
\beta_\lambda(s+1) &=& e^\infty\\
\end{eqnarray*}

And, if we go by inspection, we must have,

\[
\beta_\lambda(1 + \frac{(2k+1)\pi i}{\lambda})\,\,\text{are simple poles}\\
\]

But,

\[
\beta_\lambda(j + \frac{(2k+1)\pi i}{\lambda})\,\,\text{are essential singularities for}\,\,j \ge 2\\
\]

So, if we have any hope in constructing a tetration function $\tet(s)$ which is holomorphic on $s \in \mathbb{H}$, which is real-valued; satisfies $\overline{\tet(s)} = \tet(\overline{s})$; we'll need to get rid of these singularities somehow. And, we'll need a Riemann mapping on $\mathbb{L}$ to do this. Now, given, we don't need as miraculous a Riemann mapping as Kneser. We just have to avoid the poles somehow. Using the $\sqrt{\cdot}$ function will suffice.

The way we're going to do this is actually pretty simple. We're going to use a function $\lambda(s) : \{s \in \mathbb{C}\,|\,|\arg(s)|<\theta,\,0 < \theta < \pi/2\} \to \mathbb{L}_2$. Where $\mathbb{L}_2$ is a projection into the second component of $\mathbb{L}$. This can be better explained, such that $(s,\lambda(s)) \in \mathbb{L}$ is satisfied for $|\arg(s)| < \theta$. We'll require that $\lambda:\mathbb{R}^+ \to \mathbb{R}^+$. And of course that $\lambda$ is holomorphic.

From here, we'll define a sequence of functions,

\[
F_n(s) = \log^{\circ n} \beta_{\lambda(s+n)}(s+n)\\
\]

Where we use all of the above expressions from the previous sections, to prove that the error $\tau^n_{\lambda(s+n)}(s)$ converges as $n\to\infty$. And then we use this to prove that $F_n \to F$. And here, $F$ will be holomorphic for $|\arg(s)| < \theta$. But as soon as $F$ is holomorphic for $\Im(s) = A$ it is necessarily holomorphic for $\Im(s-j) = A$ by the functional relationship and the non-zero nature of $F$. And so, this will construct a tetration function $F(s) : \mathbb{C}/(-\infty,X] \to \mathbb{C}$ which is real-valued.

Hereupon we take the value $x_0 \in (X,\infty)$ in which $F(x_0) = 1$ as a transfer value so that $F(s+x_0) = \tet(s)$; which will be our desired tetration. Which is quite the mouthful; but summarizes pretty clearly our line of attack.

So we write the desired theorem,

\begin{theorem}[The Desired Mapping Theorem]
The function ${\displaystyle \lambda(s) = \frac{1}{\sqrt{1+s}}}$, is holomorphic for $|\arg(s)| < \theta$ for some $0< \theta < \pi/2$, and $\lambda:\mathbb{R}^+ \to \mathbb{R}^+$, where the pair $(s,\lambda(s)) \in \mathbb{L}$ for $|\arg(s)| < \theta$. 
\end{theorem}

We'll prove this in a bit. First we need to better understand the family of functions $\beta_\lambda(s)$.

\section{Normality theorem at infinity; removing the singularity}\label{sec5}
\setcounter{equation}{0}

When discussing the convergence of $\beta_\lambda(s)$ so far; we've only referred to convergence on compact sets. When discussing infinite compositions at infinity, the author points the reader to \cite{Nix}, where asymptotics were used on infinite compositions. Where the general idea was, we had uniform convergence wherever the sum converges; even at infinity.

This can be seen in the statement of Theorem \ref{thmA}. The set $\mathcal{S}$ is a set where the sum $\sum_{j=1}^\infty ||H_j(s,z) - A||_{s\in\mathcal{S}} < \infty$--and the infinite composition converges uniformly for all $s \in \mathcal{S}$. The point at infinity can be included in $\mathcal{S}$ if the sum converges uniformly on $\mathcal{S}$.

Ipso, if we call a $\mathcal{U} \subset \mathbb{L} \cup \infty$ in which $(s,\lambda) \in \mathcal{U}$ satisfy,

\[
\sum_{j=1}^\infty \frac{1}{||e^{\lambda(j-s)} +1||_{\mathcal{U}}} < \infty
\]

Then, $\beta_\lambda(s)$ is continuous on $\mathcal{U}$, including the point at infinity (it becomes a removable singularity). In this case, conveniently, we can think of the point at infinity as $\Re(s) \to -\infty$ or $\Re(s) \to \infty$. And so, in a half plane of $\mathbb{L}$ we have the identification that,

\[
\beta(-\infty) = 0\\
\]

Because when $\Re(s) = -\infty$ we have that,

\[
\lim_{\Re(s) \to -\infty}\sum_{j=1}^\infty |\frac{e^z}{e^{\lambda (j-s)} +1}| = \sum_{j=1}^\infty \lim_{\Re(s) \to -\infty}|\frac{e^z}{e^{\lambda (j-s)} +1}| = 0\\ 
\]

And this can be done compactly, where if we include this point at infinity, we know that the sum converges uniformly. And additionally that $\beta_\lambda(s)e^{-\lambda (s-1)} = 1$ as $\Re(s) \to -\infty$. This was explored deeply in \cite{Nix}; where we derived asymptotic relationships for infinite compositions.

This transformation can be thought similarly to how we've constructed $g_\lambda(w)$. The value of $g_\lambda(0) = 0$ is the value of $\beta_\lambda(s)$ at $(s,\lambda) = -\infty$; and the value of $g_\lambda(\infty) = \infty$ is the values of $\beta_\lambda(s)$ at $(s,\lambda) = \infty$. This means, exactly like $e^{\lambda s}$ our function $\beta_\lambda(s) \not\to 0$ as $(s,\lambda) \to \infty$, but tends to zero as $(s,\lambda) \to -\infty$. Since $\mathbb{L}$ is in two complex dimensions, we can assign two points at infinity. In such a sense, we are considering $\widehat{\mathbb{L}} = \mathbb{L}\cup \{\infty,-\infty\}$

Now, if we take the $g_\lambda(w)$ approach at understanding $\widehat{\mathbb{L}}$, the real trouble is when we look at the real $\infty$ in question (when $\Re(s) = \infty$, $w=\infty$). This is when $g_\lambda(w)$ is in a neighborhood of infinity; which is when our asymptotic kicks in.

In this space, we consider $g_\lambda(1/w)$, which we give by, yet again, an infinite composition. We'll begin now to denote this,

\[
f_\lambda(w) = g_\lambda(1/w) = \OmSum_{j=1}^\infty \frac{e^z}{e^{\lambda j}w +1}\,\bullet z\\
\]

And we'll attach the commutative diagram,\\

\begin{center}
\begin{tikzcd}
\beta_\lambda(s) \arrow[rr, "s \mapsto s+1"] \arrow[dd, "s = -\log(w)/\lambda"] &  & \frac{e^{\beta_\lambda(s)}}{e^{-\lambda s} + 1} \arrow[dd, "s = -\log(w)/\lambda"] \\
                                                                                &  &                                                                                    \\
f_\lambda(w) \arrow[rr, "w \mapsto e^{-\lambda} w"]                             &  & \frac{e^{f_\lambda(w)}}{w+1}                                                      
\end{tikzcd}
\end{center}

This function is holomorphic on $\mathbb{C}^\times=\{w \in \mathbb{C}\,|\,w\neq 0,\,w \neq-e^{-\lambda j}\}$; but we don't need such an expansive domain. Instead we'll focus on the unit disk $\mathbb{D}$ subtracting our bad points. The set in question is,

\[
\mathbb{D}^\times = \{w \in \mathbb{D}\,:\,0 < |w| < 1,\,w \neq - e^{-\lambda j}\}\\
\]

And on this punctured disk (sort of, forgive the abuse),

\[
f_\lambda(e^{-\lambda}w) = \frac{e^{f_\lambda(w)}}{w+1}\\
\]

Where in this space,

\[
\log f_\lambda(e^{-\lambda}w) = f_\lambda(w) - \log(1+w)\\
\]

This is again understood as a commutative diagram,\\

\begin{center}
 \begin{tikzcd}
f_\lambda(w) \arrow[rr, "w \mapsto e^{-\lambda} w"] \arrow[dd, "z \mapsto \log z"] &  & \frac{e^{f_\lambda(w)}}{w+1} \arrow[dd, "z \mapsto \log z"] \\
                                                                                                   &  &                                                                             \\
\log f_\lambda(w) \arrow[rr, "w \mapsto e^{-\lambda} w"]                                           &  & f_\lambda(w) -\log(1+w)                                                    
\end{tikzcd}   
\end{center}

Where here, all our discussions of convergence are linearized. We have the contraction $\mathcal{T} : p(w) \mapsto p(e^{-\lambda}w)$ and a sum $z\mapsto z+\log(1+w)$; and our exponential $e^{-\lambda s} \mapsto w$. What we want to say, which is incredibly simple in this space, is that this asymptotic relationship is satisfied for $\mathbb{D}^\times \cup \{0\}$. It isn't much, but the holomorphy of tetration follows from this.\\

\begin{figure*}
  \centering
  \includegraphics[scale=0.6]{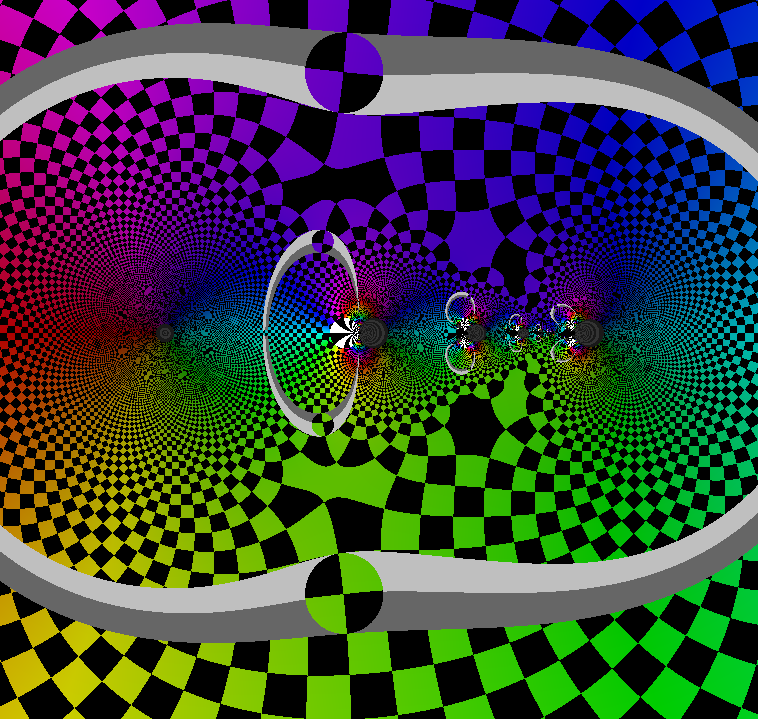}\quad
  \hfill
  \caption {The function $f_{\log(2)}(w) = g_{\log(2)}(1/w)$ about $w=0$ for $n=10$ iterations. For $w \in \mathbb{R}_{x\le 0}$ the function exhibits its most extreme behaviour--we can see the essential singularities pretty clearly. To the right of these (on the right side of the graph) you can see a nice convergence to infinity as $w\to0$.}
  \label{figlog}
\end{figure*}

\begin{figure*}
  \centering
  \includegraphics[scale=0.6]{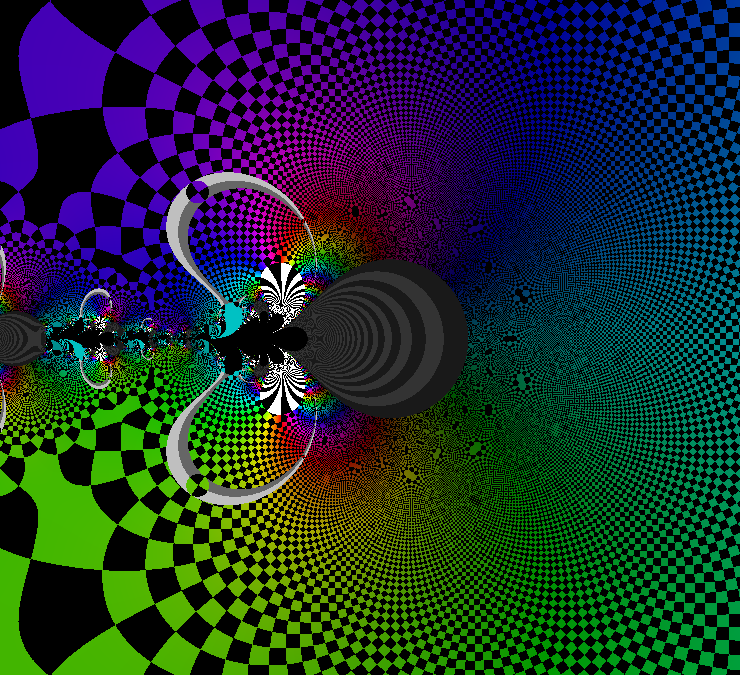}\quad
  \hfill
  \caption {The function $f_{\log(2)}(w) = g_{\log(2)}(1/w)$ about $w=0$ for $n=10$ iterations. We can see its behaviour near zero is fairly regular to the right of the singularities. As it grows super-exponentially to infinity. As we increase the depth of the iteration, this area grows to nearly encircle $0$, excepting $\mathbb{R}_{x \le 0}$.}
  \label{figlog2}
\end{figure*}

\begin{figure*}
  \centering
  \includegraphics[scale=0.6]{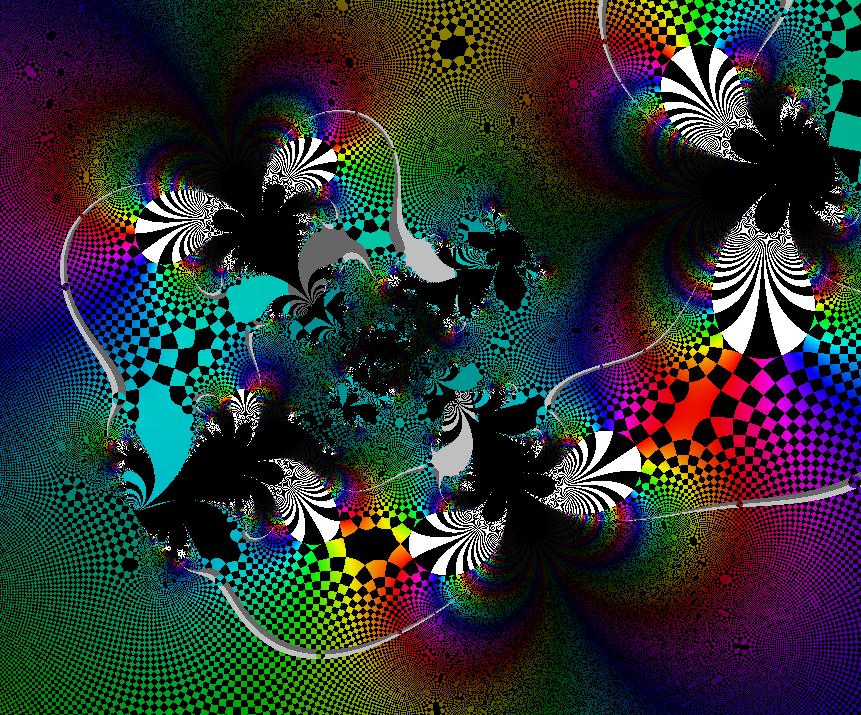}\quad
  \hfill
  \caption {The function $f_{\lambda}(w) = g_\lambda(1/w)$ about $w=0$ for $n=10$ iterations; $\lambda = 1/2 + 3i$. As chaotic as this graph looks; it's precisely as well behaved as the previous figures.}
  \label{fig4}
\end{figure*}

In Figure \ref{figlog}, we see a regular kind of structure. Everything looks fairly uniform in a neighborhood of zero. This function gets closer and closer to satisfying the equation $f(\frac{w}{2}) = e^f$--while satisfying $f(\frac{w}{2}) = \frac{e^{f(w)}}{1+w}$. Nonetheless, it looks fairly regular.

In Figure \ref{figlog2}, we get a closer look at it diverging to infinity. Where, we can expect iterated $\log$'s to converge uniformly--because $f_{\log(2)}$ stays very very large.

In Figure \ref{fig4} we can see the absolute chaos that happens near $0$. Where the essential singularities begin to converge towards $0$. However, since we're only $10$ iterations deep, there are only $9$ essential singularities. We should expect worse in the final result--but also better because at $n=\infty$ (the depth of iteration) the functional equation is satisfied.

Nonetheless, do not be put off by Figure \ref{fig4}; it is exactly as regular as Figure \ref{figlog} and Figure \ref{figlog2}. It just looks insanely more complex, because we've used a complex multiplier, rather than a real multiplier. The mapping argument we make works the exact same way in both scenarios.\\

Despite the chaos; the following identity is satisfied for all $\mathbb{D}^{\times}$ and since $\log(1+w)$ is holomorphic on $\mathbb{D}$ we know that,

\[
\log f_\lambda(e^{-\lambda} w) - f_\lambda(w) = -\log(1+w) \sim -w\,\,|w|\to 0\\
\]

And when we pull back into the space $\mathbb{L}$, it means that,

\[
\log \beta_\lambda(s+1) - \beta_\lambda(s) = 0\,\,\text{when}\,\, (s,\lambda) = +\infty\\
\]

And it means this uniformly. In such a sense we can assign a point at $\pm\infty$ in $\mathbb{L}$ in which $\beta(-\infty) = 0$ and $\beta(\infty) = \infty$; if we keep this with the correlation that $\pm \infty = \Re s$; and at infinity the statement $\log(\beta(\infty + 1)) - \beta(\infty) = 0$ is actually pretty meaningful. Whereupon, in $\widehat{\mathbb{L}}$ this is fully rigorous $\log(\beta(\infty)) = \beta(\infty)$.

It's important to remember we are absolutely not talking about the points $\Im s = \pm \infty$. We have absolutely nothing to say about these points when we make the change of variables back to $s$. We are only focused on shifting the real argument to the left or the right. In fact, in the complex plane this will be much much uglier.

To translate this into a more standard way of thinking $\beta_\lambda(s) \to \infty$ as $\Re(s) \to \infty$, but not necessarily on other paths towards infinity. Since we are keeping $\lambda$ finite, these are \textit{sort of} equivalent statements. We may get ourselves in a bind if we push this isomorphism to the extreme, but it works fine for what we need it for. 

Now when we talk about our functions $\tau_\lambda^n(s) = u_\lambda^n(w)$ in this different correspondence; everything is straight forward. Once everything is linearized we get that,

\[
u_\lambda^{n+1}(w) = \log(f_\lambda(e^{-\lambda}w) + u_\lambda^n(e^{-\lambda}w)) - f_\lambda(w)\\
\]

Each $u_\lambda^n(w)$ is holomorphic on $\delta\mathbb{D}^\times = \{0 < |w| < \delta,\,w \neq -e^{-\lambda j}\}$ for an appropriately small $\delta$ and has a removable singularity at $0$; which is a fixed point $u_\lambda^n(0) = 0$ with multiplier $u_\lambda^{n\prime}(0) = -1$. We can relate this to The Removable Singularity Theorem \ref{thmRMV}; which is,

\[
u_\lambda^n(w) = -\log(1+w) + o(w)\,\text{as}\,w \to 0^*\\
\]

For all $w \in \delta \mathbb{D}^\times$--*if we take this limit appropriately. The manner to take the limit is as $w\to 0 \mapsto e^{-\lambda j} w$ as $j\to\infty$. Which is to say, the limit must be taken as a spiral into zero, controlled by our multiplier $e^{-\lambda}$. This translates into right translations of $s \to s+1$, when we make a change of variables.

This observation is very important because $f_\lambda(w)$ has an essential singularity at $0$. We want to spiral into this essential singularity in a certain way, such that our asymptotics can take over. There's nothing we can say about alternative paths. In fact, there will always be a path where this construction fails; it's an essential singularity; Picard's Theorem is king.

The $o$ term, by The Removable Singularity Theorem \ref{thmRMV}, can be expressed as,

\[
o(w) = w\sum_{j=1}^{n-1} \frac{q^j}{B_j(w)}\\
\]

For a number $e^{-\Re \lambda} < q < 1$; and for the sequence of functions,

\[
B_j(w) = \prod_{k=1}^j f_\lambda(e^{-\lambda k}w)
\]

This implies we have a normality condition on $u_\lambda^n(w)$. Which is to say, that on a compact set $\mathcal{B} \subset \mathbb{D}^\times$, the function $||u_\lambda^n(w)||_{\mathcal{B}} \le M$ for some $M \in \mathbb{R}^+$. So we can expect the sequence of functions $u_\lambda^n$ to be bounded.

But we have something better than that. By The Normality At Infinity Theorem \ref{thmNRMINF} we know that,

\[
||u_\lambda^n(w) +\log(1+w)||_{\mathcal{B}} \le D \left | \left |\sum_{j=1}^\infty \dfrac{wq^j}{\prod_{k=1}^j |f_\lambda(e^{-\lambda k}w)|}\right | \right |_{\mathcal{B}}\\
\]

For $\mathcal{B}$ a compact set for $w\in\mathbb{D}^\times$ and $\Re \lambda > 0$. This can be done uniformly for both $w$ and $\lambda$--and the value $q$ is as it was before. This gives us a clear normality condition on the sequence $u_\lambda^n$, just as we have for $\tau_\lambda^n$.\\

Again, thinking of this as a commutative diagram is perfectly possible. We're trying to develop the limit of this commutative diagram. 

\begin{center}
\begin{tikzcd}
f_\lambda + u_\lambda^n \arrow[rr, "w \mapsto e^{-\lambda} w"] \arrow[dd, "z \mapsto \log (z)"] &  & e^{f_\lambda + u_\lambda^{n+1}} \arrow[dd, "z \mapsto \log(z)"]                                 &  &                                    \\
                                                                                                                                          &  &                                                                                                 &  &                                    \\
\log f_\lambda + u_\lambda^n \arrow[rr, "w \mapsto e^{-\lambda} w"]                                                                       &  & f_\lambda + u_\lambda^{n+1} \arrow[rr, "n \to \infty", dashed] \arrow[dd, "n\to\infty", dashed] &  & {} \arrow[dd, "z \mapsto \log(z)"] \\
                                                                                                                                          &  &                                                                                                 &  &                                    \\
                                                                                                                                          &  & {} \arrow[rr, "w \mapsto e^{-\lambda} w"]                                                       &  & \varphi_\lambda                   
\end{tikzcd}
\end{center}

Where, the limit of this commutative diagram is a map $\varphi_\lambda = f_\lambda + \lim_{n\to\infty} u_\lambda^n = f_\lambda + u_\lambda$; which satisfies the diagram,

\begin{center}
\begin{tikzcd}
\varphi_\lambda \arrow[rr, "w \mapsto e^{-\lambda} w"] \arrow[dd, "z \mapsto \log(z)"] &  & e^{\varphi_\lambda} \arrow[dd, "z \mapsto \log(z)"] \\
                                                                                       &  &                                                     \\
\log \varphi_\lambda \arrow[rr, "w \mapsto e^{-\lambda} w"]                            &  & \varphi_\lambda                                    
\end{tikzcd}
\end{center}

Now, if we consider the contraction map $\mathcal{T}$ on functions $y$ holomorphic on $\delta \mathbb{D}^\times$ and of the form $y(w) = -\log(1+w) + o(w)$ (where the $o$ term is bounded in $n$); then,

\[
\mathcal{T} y_1 - \mathcal{T} y_2 = y_1(e^{-\lambda}w) - y_2(e^{-\lambda} w) = o(e^{-\lambda} w)\\
\]

Thus, looks like $e^{-\lambda}wp(w)$ where $p(w) \to 0$ as $w \to 0$. Which implies that, on a compact set $\mathcal{J} \subset \delta \mathbb{D}^\times$ that,

\[
||\mathcal{T}y_1 - \mathcal{T}y_2||_{w \in \mathcal{J}} \le q ||y_1 - y_2||_{w\in \mathcal{J}}\\
\]

For $|e^{-\lambda}| < q < 1$; which shrinks to $|e^{-\lambda}|$ as $\delta \to 0$. So, on this space we can call $\mathcal{T}$ a contraction mapping. We can better understand why we need this, by looking at the formula,

\[
u_\lambda^{n+1}(w) - u_\lambda^n(w) = \log\left( \frac{f_\lambda(e^{-\lambda}w)+ u_\lambda^{n}(e^{-\lambda} w)}{f_\lambda(e^{-\lambda}w)+ u_\lambda^{n-1}(e^{-\lambda} w)}\right)\\
\]

Then as, $u(e^{-\lambda}w)/f_\lambda(e^{-\lambda}w) \to 0$ under the spiral $e^{-\lambda j}$; we can expect a Lipschitz constant $A \to 1$.

\begin{eqnarray*}
|u_\lambda^{n+1}(w) - u_\lambda^n(w)| \le A|u_\lambda^n(e^{-\lambda}w) - u_\lambda^{n-1}(e^{-\lambda}w)|\\
\end{eqnarray*}

Which is better understood,

\[
\left|\left|\log\left(\frac{f_\lambda(w) + w'}{f_\lambda(w) + w}\right)\right| \right|_{\mathcal{B}} \le A|w-w'|\\ 
\]

With a fixed point at $w,w'=0$. The constant $1 < A < 1+\epsilon$ is eventual for $|w|,|w'| < \delta$; where $\epsilon \to 0$ as $\delta \to 0$; because $\frac{u_\lambda^n(w)}{w} \to -1$ and $\frac{1}{f(e^{-\lambda k}w)} \to 0$ as $k\to\infty$. This implies the Lipschitz constant satisfies at least $A \to 1$. A better estimate would be $A = \frac{1}{f_\lambda(e^{-\lambda j}w)} = o(1)$--but we have such slow decay, $A\to1$ needs to be enough.

Now if we take a compact set,

\begin{eqnarray*}
\mathcal{B}_{\delta,\delta'} &=& \{(w,\lambda) \in \mathbb{C}^2\,|\,|w| \le \delta,\, |w + e^{-\lambda j}| \ge \delta' e^{-\Re \lambda j},\,j\ge 1,\\
&& 0< a\le \Re\lambda \le b,\,c \le \Im \lambda \le d\}\cup\{0\}\\
\end{eqnarray*}

Then $||u_\lambda^{n}(e^{-\lambda}w) - u_\lambda^{n-1}(e^{-\lambda} w)||_{\mathcal{B}} \le q||u_\lambda^n(w) - u_\lambda^{n-1}(w)||_{\mathcal{B}}$ for $e^{-\Re\lambda} \le e^{-a} < q < 1$.  And therefore $u_\lambda^n$ converges uniformly as $n\to\infty$ on the compact set $\mathcal{B}$ because $0 < Aq< 1$. And this is precisely our tetration existence theorem. Where upon $u_\lambda(w) : \mathcal{B} \to \mathbb{C}$, in which the functional equation $f_\lambda(e^{\lambda}w) + u_\lambda(e^{\lambda}w) = \varphi_\lambda(e^\lambda w) = \log\varphi_\lambda =  \log(f_\lambda + u_\lambda)$ extends this function almost everywhere.

 Now, the compact set $\mathcal{B}$ depends on $\delta$ and $\delta'$, and we can shrink $\delta$ or $\delta'$. The value $\delta$ controls our Lipschitz constant $A$ and our contracting constant $q$. And the value $\delta'$ controls how close we allow ourselves to the essential singularities at $w = -e^{-\lambda j}$. Our contracting constant $q$ can get as close as we want to $e^{-a}$; similarly with $A \to 1$. But additionally, so long as $e^{-\Re \lambda} \le e^{-a}$, we can alter $\delta$ and $\delta'$ uniformly in $\lambda$ in which $qA< 1$. And as such, the convergence is uniform in $\lambda$ as well as $w$, because $(\delta,\delta')$ can be chosen uniformly in $\lambda$.\\

\begin{theorem}[Tetration Existence Theorem]\label{thmTE}
For $(s,\lambda) \in \mathcal{L} \subset \mathbb{L}$ there exists a holomorphic tetration function $F_\lambda(s)$ such that,

\[
F_\lambda(s+1) = e^{F_\lambda(s)}\\
\]

Where $\mathbb{L}/\mathcal{L}$ is a measure zero set in $\mathbb{C}^2$.
\end{theorem}

\begin{proof}
The correct way to observe this is a bit more tacit. This is the quick run through of everything we've done above. We start with the result,

\[
\Big{|}\log\big{(}\frac{f_\lambda(w) + w}{f_\lambda(w) + w'}\big{)}\Big{|} \le A|w - w'|\\
\]

Where the smaller $|w-w'| < \delta$ is, the closer $1 < A < 1+\epsilon$ shrinks to $\epsilon = 0$. Because $1/f_\lambda(w) = o(1)$ as $w \to 0^*$--*in a spiral--this is guaranteed. And so the recursive process we've derived,

\[
u_\lambda^{n}(w) = -\log(1+w) + o(w)\\
\]

Because the initial convergent is $u_\lambda^{1}(w) = -\log(1+w)$, and since the further iterates will converge to $u_\lambda^1(w)$ as $w\to 0^*$ we know $u_\lambda^n(w)/w \to -1^*$. Hence, $A \to 1$ at least. For the second half of the proof; parce que,

\[
u_\lambda^{n}(e^{-\lambda}w) = -\log(1+e^{-\lambda}w) + o(e^{-\lambda}w)\\
\]

And we can bound, for some $D \in \mathbb{R}^+$, $|u_\lambda^n(w)| \le w(1+D)$. So, we know that the operator,

\[
\mathcal{T} u_\lambda(w) = u_\lambda(e^{-\lambda}w)\\
\]

Is a contraction, subject to Banach's theorem; with a contraction constant $e^{-\Re\lambda} < q < 1$ for $q$ dependent on a compact set. This theorem can be applied uniformly. So choose a compact set $\mathcal{B} =\{(w,\lambda)\in\mathbb{C}^2\,|\,0<|w| \le \delta,\,|w+e^{-\lambda j}| \ge \delta'e^{-\Re \lambda j},\,j\ge1,\,0< a\le \Re\lambda \le b,\,c\le \Im \lambda \le d\}\cup\{0\}$, mais soi,

\[
||u_\lambda^{n}(e^{-\lambda}w) - u_\lambda^{n-1}(e^{-\lambda}w)||_{\mathcal{B}} \le q||u_\lambda^{n} - u_\lambda^{n-1}||_{\mathcal{B}}\\
\]

And from this, we can choose $e^{-a} < q < 1$ for $\delta>0$. Choose $\delta,\delta'$ such that $0< qA = q(1+\epsilon) < 1$, and then,

\begin{eqnarray*}
||u_\lambda^{n+1} - u_\lambda^n||_{\mathcal{B}} &\le& A||u_\lambda^n(e^{-\lambda}w) - u_\lambda^{n-1}(e^{-\lambda}w)||_{\mathcal{B}}\\
&<& qA||u_\lambda^{n} - u_\lambda^{n-1}||_{\mathcal{B}}\\
\end{eqnarray*}

Which concludes the convergence of $u_\lambda^n(w) = \tau_\lambda^n(s)$ by Banach's Fixed Point Theorem. The function $\tau_\lambda(s) = \log(\beta_\lambda(s+1) + \tau_\lambda(s+1)) - \beta_\lambda(s)$; upon which we can iterate this process to $s \in \mathbb{C}$ for all $\tau_\lambda(s)$, except for branching points. Additionally, this proof was done in a uniform manner so that we have local holomorphy in $\lambda$; giving us the theorem.
\end{proof}

To better explain the situation with $\lambda$, we suggest the reader look at John Milnor's \cite{Mil} treatment of the holomorphy of the Schr\"{o}der function in its multiplier. Which is a limit process described by a multiplier, where the result is holomorphic in the multiplier; which is what we have here. Except the multiplier is written $e^{-\lambda}$.

Milnor uses a similar summation argument--attributed as Koenig's Linearization Theorem; where the majorant for the sum is locally independent of $\lambda$. This is precisely the case here; a sum is bounded uniformly in $\lambda$--such $e^{-\Re \lambda} \le e^{-a} < q < 1$; and $a$ acts as a limiter for $\lambda$. And $a>0$ can be chosen with $\delta,\delta'$.

Now this theorem isn't exactly what we want. We want to take the limit $\lim_{n\to\infty} u_\lambda^n(w)$ as $\lambda$ depends on $w$. This is perfectly doable by the above analysis; since convergence is uniform in $\lambda$ and $w$. But we need $u_\lambda^n(w)$ to converge in a manner where $\lambda \to 0$ as $n\to\infty$; this is not covered by this theorem. We want a mapping; a nice enough function $\lambda \mapsto \lambda(s)$ to wash away all the problems. It has to tend to $0$ like $n^{-\epsilon}$ for $\epsilon > 0$ at least, and $\epsilon < 1$ at most.\\

\begin{figure}
  \centering
  \includegraphics[scale=.5]{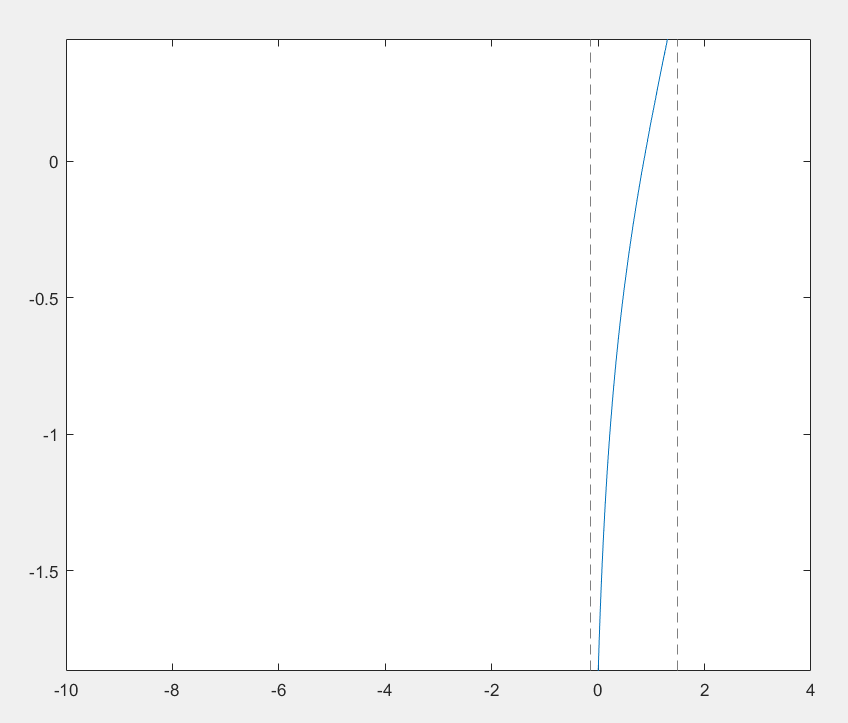}\quad
  \hfill
  \caption {The function $\log(\log(\log(\log(\beta_{log(2)}(x+4)))))$; where we can see it converging to a tetration function over $[-2,-1]$ minus a shifting argument. It's difficult to do this better without hitting overflow arguments.}
  \label{figTETLOG}
\end{figure}

\begin{figure}
    \centering
    \includegraphics[scale=.6]{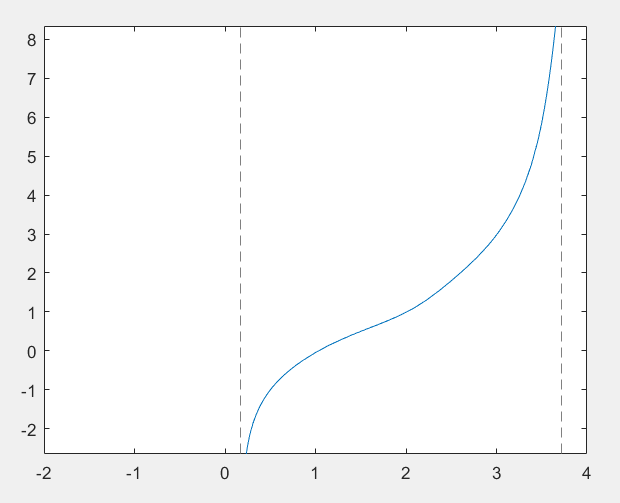}
    \caption{This is a better graph of what $F_{\log(2)}$ looks like using a sequential approach at defining the error term $\tau$.}
    \label{figTETSEQ}
\end{figure}

We'd like to spend some time visualizing this theorem. It may be difficult to see what's going on exactly. For that reason, it's helpful to use a geometric interpretation. But sadly, as we are dealing with iterated exponentials at infinity, overflow errors abound. The author still has no way of graphing this function without overflow errors. Which in many ways, is expected to happen--as we're pulling back a super-exponential function at $\infty$.

The function $f_\lambda(w)$ has essential singularities at the points $w = -e^{-\lambda j}$ accumulating at the essential singularity at $w=0$. However, if we spiral into this singularity like $e^{-\lambda n}w$; we have good control over it. Such that,

\[
u_\lambda^n(w) = \log^{\circ n}f_\lambda(e^{-\lambda n} w) - f_\lambda(w)\\
\]

Will have very good control, and we have convergence. If we choose to do this in another spiral, or in another arbitrary path to $0$; there's no guarantee as to what will happen. It may explode to infinity, or shrink astronomically to $0$. But, along this spiral all is good.\\

\begin{figure}
  \centering
  \includegraphics[scale=.5]{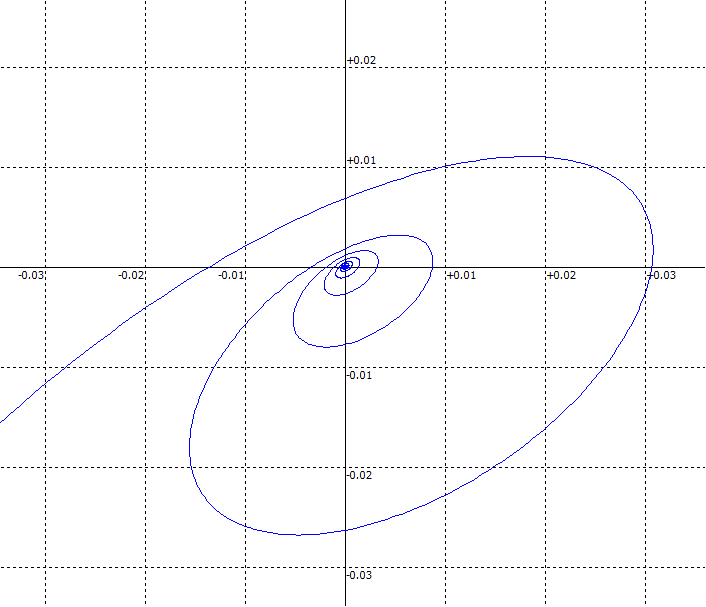}\quad
  \hfill
  \caption {A spiral in question, for $w=1/2 + 3/4i$ and $\lambda = (1+i)\log(2)$.}
  \label{fig5}
\end{figure}

\begin{figure}
  \centering
  \includegraphics[scale=.6]{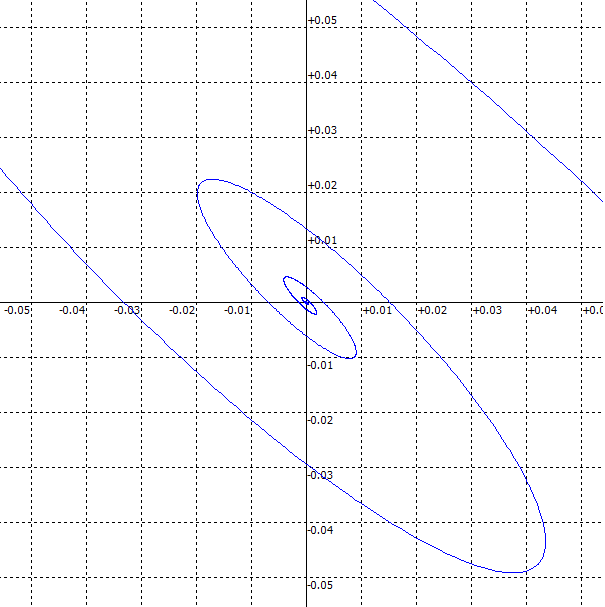}\quad
  \hfill
  \caption {A spiral in question, for $w=-1/3 + 1/5i$ and $\lambda = 1/2 + 2i$.}
  \label{fig6}
\end{figure}

The attached spirals in Figure \ref{fig5} and \ref{fig6} essentially describe what the general shape will look like. It's really nothing more than an odd shaped logarithmic spiral. And we want to trace along these spirals into $0$. Where in such manners $\log^{\circ n} f_\lambda(e^{-\lambda n}w) - f_\lambda(w) \to u_\lambda(w)$ as $n\to\infty$.

This works for two reasons. The first being that the operator $\mathcal{T} p(w) = p(e^{-\lambda} w)$ is a contraction mapping. The second being, $\frac{u(e^{-\lambda n}w)}{e^{-\lambda n}w} \to -1$ as $n\to\infty$--which is covered by The Removable Singularity Theorem \ref{thmRMV}. We can think of the spiral as a right-translation in our change of variables.

Additionally, along this spiral, our function $f_\lambda(w)$ satisfies a natural functional equation,

\[
f_\lambda(e^{-\lambda}w) = \frac{e^{f_\lambda(w)}}{1+w}\\
\]

Where, for other spirals--we are clueless as to how our function $f_\lambda$ behaves. All of this surgery means--along this spiral, we can artfully remove the singularity at $0$. And at $0$; our tetration functional equation gets closer and closer to being satisfied. So long as we pay attention to the other singularities which happen along the spiral $w = - e^{-\lambda n}$; and we make sure to stay somewhat away from these. Which is precisely what our compact set $\mathcal{B}$ does.

If we write,

\begin{eqnarray*}
\mathcal{B}_{\delta,\delta'} &=& \{(w,\lambda) \in \mathbb{C}^2\,|\,|w|\le\delta,\,|w+e^{-\lambda j}| \ge \delta'e^{-\Re \lambda j},\,j\ge1,\\
&&0<a \le \Re \lambda \le b,\,c\le \Im \lambda \le d\}\cup\{0\}\\
\end{eqnarray*}

Then $\mathcal{B}$ is chosen so that $(e^{-\lambda} w,\lambda) \in \mathcal{B}$ if $(w,\lambda) \in \mathcal{B}$. And so our spiral sits well in this set. We can visualize this in $w$ as a compact disk about $0$ of radius $\delta$--and each point $w = - e^{-\lambda j}$ has a small disk about it (with a radius $\delta'e^{-\Re \lambda j}$); in which these disks are excluded from the compact disk about $0$. So we can think of this as a disk with a bunch of smaller and smaller holes spiraling into $0$. Where the hole at $0$ is a removable singularity. Figure \ref{fig7} displays the general shape.\\

\begin{figure}
  \centering
  \includegraphics[scale=.6]{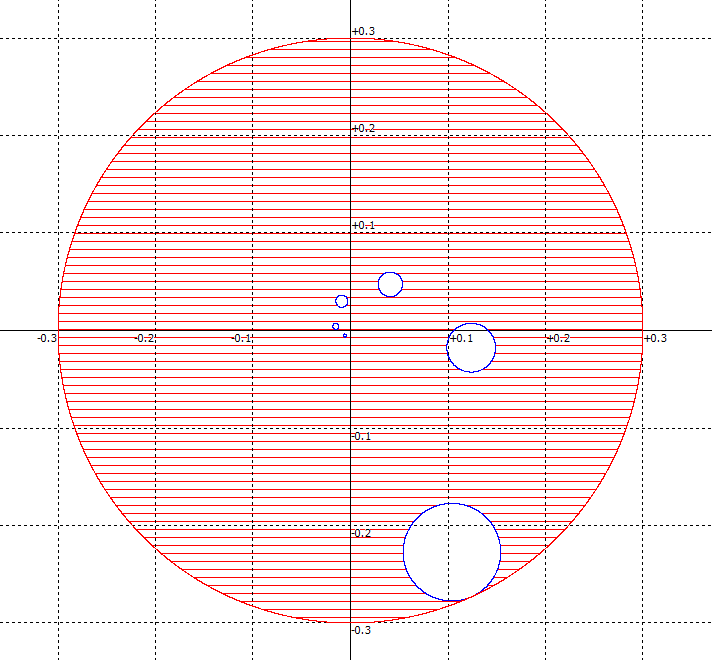}\quad
  \hfill
  \caption {The set $\mathcal{B}_{\delta,\delta'}$ for $\delta = 0.3$, $\delta' = 0.2$, $\lambda = \log(2) - i$. The shaded area is the domain $\mathcal{B}$. The singularities of $f_\lambda(w)$ occur in the center of each disk.}
  \label{fig7}
\end{figure}

\section{Finding an appropriate mapping; and generalizing Theorem \ref{thmTE}}\label{sec6}
\setcounter{equation}{0}

It's helpful to look at this problem from two different angles. The first being a mapping $\lambda(w) : \mathbb{D}^\times \to \Re(\lambda) > 0$ and the second being a mapping $\lambda: \mathbb{L}_1 \to \mathbb{L}_2$. Recall that $\mathbb{D}^\times = \{w \in \mathbb{C}\,|\,0 < |w| < 1,\,w \neq -e^{\lambda j},\,j \ge 1\}$ and $\mathbb{L} = \mathbb{L}_1 \times \mathbb{L}_2 = \{(s,\lambda) \in \mathbb{C}^2\,|\, \Re(\lambda) > 0,\,\lambda(j-s)\neq (2k+1)\pi i,\,j\ge1,\,k\in\mathbb{Z}\}$.

In the second form, we want a mapping $\lambda : \{s \in \mathbb{C}\,|\, |\arg(s)| < \theta < \pi/2 \} \to \mathbb{L}_2$ in which $\lambda : \mathbb{R}^+ \to \mathbb{R}^+$. We want to convert this into a restriction on $\lambda(w)$. The answer to this riddle isn't too difficult.

To find an appropriate mapping, we just want a mapping that expands the lines where,

\[
\lambda(j-s) = (2k+1)\pi i\\
\]

This can be done with the function $\sqrt{s}$, and many functions like this will work, but $\sqrt{s}$ is simple enough. To visualize this, when $\lambda$ is constant we have a lattice of points in the right half plane. And if we were to multiply that lattice by $\frac{1}{\sqrt{1+s}}$ we'd be able to place a sector $|\arg(s)| < \theta$ within it. Or rather,

\[
\frac{j-s}{\sqrt{1+s}} = (2k+1)\pi i\\
\]

Then,

\[
s = j + \sqrt{1+s}2(k+1)\pi i
\]

And there's a sector $|\arg(s)| < \theta < \pi/2$ in which this can't be true for $j \ge 1$. So with this we're going to consider the function,

\[
\beta(s) = \OmSum_{j=1}^\infty \frac{e^z}{e^{\frac{j-s}{\sqrt{1+s}}} + 1}\,\bullet z\\
\]

And it's alternate form on $\mathbb{D}^\times$,

\[
f(w) = \OmSum_{j=1}^\infty \frac{e^z}{e^{\frac{j}{\sqrt{1+s}}} w +1}\bullet z\\
\]

But,

\begin{eqnarray*}
w &=& e^{-\lambda s}\\
s &=& -\log(w)/\lambda\\
s &=& -\frac{\log(w)}{\sqrt{1+s}}\\
\end{eqnarray*}

So that $h(w) = \frac{1}{\sqrt{1+s}}$ for some holomorphic function $h$,

\[
f(w) = \OmSum_{j=1}^\infty \frac{e^z}{e^{h(w)j} w +1}\bullet z\\
\]

Now, this expression looks very cryptic. But we know that $u_\lambda^n(w)$ as $n\to\infty$ converges uniformly for $(w,\lambda) \in \mathcal{B}_{\delta,\delta'}$--where this is a compact set. When we translate this back into $\beta_\lambda(s)$; where $\lambda = \frac{1}{\sqrt{1+s}}$ we get that, first of all $\beta(s) : \{|\arg(s)| < \theta < \pi/2\} \to \mathbb{C}$; and secondly that,

\[
\tau^n(s) = \log^{\circ n}\beta(s+n) - \beta(s) \to 0\\
\]

Should converge uniformly for $|\arg(s)| < \theta < \pi/2$ as $|s|\to\infty$ at least like $\mathcal{O}(|s|^{-1/2})$. As such we still have our identity $\log \beta(\infty) - \beta(\infty) = 0$. Where this again equates to a compact set in $\widehat{\mathbb{L}} = \mathbb{L}\cup\{-\infty,\infty\}$.

If we think of our spirals $w_n = e^{-\lambda n} w$; we are modifying the spirals with the equivalence $\lambda = h(w)$; and letting $\lambda \to 0$ as $w_n \to 0$. This must be done very carefully, but can be done using nothing more than Banach's Fixed Point Theorem. As we have control over the multiplier, it's very possible.\\

This will imply we have a tetration function $F$ which satisfies,

\[
F(s+1) = e^{F(s)}\\
\]

For all $|\arg(s)| < \theta$. But then, for every $s \in \mathbb{C}$ there exists some $n$ such that $s + n$ is in this sector, therefore we can undo this by taking $n$ logarithms. This will define our tetration function almost everywhere--excluding logarithmic branch-cuts/singularities. Since it will be real valued, we know that there is a real value $x_0 \in \mathbb{R}$ such that,

\begin{eqnarray*}
\tet(s) &=& F(s+x_0) : (-2,\infty) \to \mathbb{R}\,\,\text{bijectively}\,\,\\
\tet(0) &=& 1\\
\end{eqnarray*}

This gives us a tetration function which definitely has singularities at the negative integers and isn't holomorphic on the line $(-\infty,-2]$. But there may be other singularities which appear elsewhere, it is our job to show this doesn't happen; which we'll do in the next section.\\

Before the dramatic conclusion, we have to take a closer look at $\tau_\lambda^n$. As you'll note, for a varying $\lambda$, $\tau_\lambda = \lim_{n\to\infty} \tau_\lambda^n$ is holomorphic; but in the case of interest, we have a sequence of $\lambda_n \to 0$ as $n\to\infty$ and $0$ is on the boundary of our domain $\mathbb{L}$. So we have to use that $\lambda_n = \mathcal{O}(n^{-1/2})$ somewhere in our construction, and show that the limit is still holomorphic in $s$. 

This is a tricky job. And of it we'll show in the following theorem.

\begin{theorem}[The Pasted Together Theorem]\label{thmPST}
The function $F_n \to F$ is holomorphic on $|\arg(s)| < \theta$; and convergence is uniform on compact sets of this.
\end{theorem}

\begin{proof}
The goal of this proof is to control the Lipschitz constant $A$ and the contraction constant $q$ from The Tetration Existence Theorem \ref{thmTE}. Now, when we shrink $\delta$ from our compact set we know that $q \to e^{-\Re\lambda}$ and that $A \to 1$, which ensures convergence. In the current situation, we know that $\Re(\lambda) \to 0$ like $\mathcal{O}(n^{-1/2})$. This implies that $\prod_{j=1}^n e^{-\Re \lambda} \to 0$ because $\sum_{j=1}^n \mathcal{O}(j^{-1/2}) \to \infty$ as $n\to\infty$, and this will be $-\infty$ because $\Re \lambda > 0$. This implies the product of our contractions looks like $e^{-Bn^{1/2}}$ for some $B>0$ which is enough to ensure convergence of the final result if we can keep $q \to e^{-\Re \lambda}$ well enough, which we can because $\delta,\delta'$ can be chosen uniformly for $\lambda$.

But in order for this work, we need $A \to 1$ in a fast enough manner to not affect this convergence. We need that $A \to 1$ like $e^{\mu n^{-1/2}}$ for $0 < \mu < 1$. Now we know we can choose $A$ such that $qA < 1$, and this choice can be done uniformly. This implies that $A \le e^{\mu \Re \lambda}$ for $0 \le \mu < 1$. 

Choose an appropriate $\delta$ and $\delta'$ such in,

\begin{eqnarray*}
\mathcal{B}_{\delta,\delta'}^n &=& \{(w,\lambda) \in \mathbb{C}^2\,|\,|w|\le\delta,\,|w + e^{-\lambda j}| \ge \delta'e^{-\Re \lambda j},\,j\ge 1,\\
\,&&0 < an^{-1/2} \le \Re\lambda \le b,\,c\le \Im \lambda \le d\}\\
\end{eqnarray*}

Our constants $q_nA_n < 1$ and $q_nA_n \to 1$ as $n\to\infty$. Note, that the compact set depends on $n$. The set,

\begin{eqnarray*}
\mathcal{B}^n \to \mathcal{U} &=& \{(w,\lambda) \in \mathbb{C}^2\,|\,|w|\le\delta,\,|w + e^{-\lambda j}| \ge \delta'e^{-\Re \lambda j},\,j\ge 1,\\
&&\,0 \le \Re\lambda \le b,\,c\le \Im \lambda \le d\}\\
\end{eqnarray*}

But by the time we hit $\mathcal{U}$, we'll be at $u^1 - u^0 = \log(1+w)$ in the inverse iteration; which has a finite value. Now, the value $\lambda$ depends on $w$; and so we can project this into a set,

\[
\mathcal{B} = \{w \in \mathbb{C}\,|\,|w|\le \delta,\,|w+e^{-\lambda j}| \ge \delta'e^{-\Re \lambda j},\,j\ge1\}\\
\]

And then, to work through this, we perform the iteration,

\begin{eqnarray*}
||u_\lambda^{n+1} - u_\lambda^n||_{\mathcal{B}} &\le& A_nq_n||u_\lambda^n - u_\lambda^{n-1}||_{\mathcal{B}}\\
&\le& A_nq_n A_{n-1} q_{n-1} ||u_\lambda^{n-1} - u_\lambda^{n-2}||_{\mathcal{B}}\\
&\vdots&\\
&\le& \prod_{j=1}^n A_jq_j ||u_\lambda^1 - u_\lambda^0||_{\mathcal{B}}\\
\end{eqnarray*}

And here,

\[
A_nq_n = e^{-(1-\mu)\mathcal{O}(n^{-1/2})}\\
\]

And we are finding such sequences, that,

\[
\prod_{j=1}^n A_jq_j = e^{-(1-\mu)\sum_{j=1}^n \mathcal{O}(j^{-1/2})} = e^{-\mathcal{O}(n^{1/2})} = e^{-Dn^{1/2}}\\
\]

For a constant $D=(1-\mu)B >0$. Therefore by induction,

\[
||u_\lambda^{n+1} - u_\lambda^n||_{\mathcal{B}} \le e^{-Dn^{1/2}} ||\log(1+w)||_{\mathcal{B}} = Me^{-Dn^{1/2}}\\
\]

Wherefore, $-\log(1+w) = u_\lambda^1 - u_\lambda^0$. For some constant $M = ||\log(1+w)||_{\mathcal{B}} = \log(1-\delta)$. Now, for all $\epsilon > 0$ there exists $N$ such for $n,m > N$ it's known,

\[
\sum_{j=n}^{m-1} e^{-Dj^{1/2}} < \epsilon/M\\
\]

And therefore,

\[
||u_\lambda^m - u_\lambda^n||_{\mathcal{B}} \le \sum_{k=n}^{m-1} ||u_\lambda^{k+1} - u_\lambda^k||_{\mathcal{B}} < \epsilon\\
\]

Which concludes the proof.
\end{proof}

We've written this theorem very quickly because we've built The Tetration Existence Theorem \ref{thmTE} to allow for a quick generalization. All that's needed to go from that theorem to this theorem is control over the Lipschitz constant $A$ and the contraction constant $q$--which we can.

This argument works exactly the same if $\lambda = \mathcal{O}(n^{-\epsilon})$ for $1>\epsilon > 0$. Further, Banach's fixed point theorem will ensure that these two processes will both converge to the same $F$. In many ways, if we take our function,

\[
F_\lambda(s) = \lim_{n\to\infty}\log^{\circ n}\beta_\lambda(s+n)\\
\]

From The Tetration Existence Theorem \ref{thmTE}; then the function $F$ we've constructed in The Pasted Together Theorem \ref{thmPST}, is essentially $F_\lambda(s) \Big{|}_{\lambda = 0}$. This is to mean, it's the boundary value of our function $F_\lambda$. However, we have to take this limit in a specific manner--not just plugging in $\lambda = 0$.

This is to imply, our constructed function is,

\[
F(s) = \lim_{n\to\infty}\lim_{\lambda \to 0} \log^{\circ n} \beta_\lambda(s+n)\,\,\text{where}\,\,\lambda = \mathcal{O}(n^{-\epsilon})\,\,\text{for}\,\,0 < \epsilon < 1\\
\]

Where, this is the gist of the construction, but obviously requires much more depth to be valid.

\section{$\tet$ is non-zero in the upper half-plane}\label{sec7}
\setcounter{equation}{0}

This section is devoted to showing that $\tet(s)$ is non-zero for $s \in \mathbb{H} = \{s \in \mathbb{C}\,|\,\Im(s) > 0\}$. This can be equivalently said, that the only zero of $\tet(s)$ is at $s = -1$. This equates to there being no singularities for $\Im(s) > 0$. 

The only way a singularity arises is if $\log \tet(s_0+1)$ is singular; which implies $\tet(s_0+1) = \infty,0$. Where we know for large enough $N$ for all $n>N$ the values $\tet(s_0 + n)$ are non-singular. Therein, the only thing that can start this chain of singularities is if $\tet(s_0 + k) = 0$ for some $k$. Upon which $\tet(s_0 +k-j) = \infty$ for all $j \ge 1$. 

So we need a theorem that $\tet$ is non-zero in the upper-half plane, and we've simultaneously showed that $\tet$ is holomorphic in the upper half-plane. By conjugation, we'll know that $\tet$ is holomorphic in the lower half-plane. And then $\tet : \mathbb{C}/(-\infty,-2] \to \mathbb{C}$ will be our maximal domain of holomorphy for $\tet$.

To do this, we need to understand what happens when $\log(\tet(s_0))= 0$. It can only happen if $\tet(s_0) = 1$, but it doesn't necessarily happen if $\tet(s_0) = 1$. There exists a curve $C$ in a neighborhood of $s_0$ in which $\tet(C) \in \mathbb{R}$. Now, supposing that $C = s_0 + t$ for $t \in (-\delta,\delta)$ then when we continue to iterate this procedure, necessarily $\tet(s_0 + t)$ will be real-valued as $t$ grows. This will force $\tet(s_0-1) = 0$. Assuming that $C$ is not a line, then the line $\tet(s_0 +t)$ is not real-valued for $t \in (-\delta,\delta)$ excepting at $t=0$. This means the logarithm $\log(\tet(s_0 + t))$ must be complex valued, and this means that $\tet(s_0 - 1)$ must be in a neighborhood of $2\pi i k$ for $k\neq 0$, and it is non-zero.

So all we have to do is focus our attention on $\tet(s_0 + t)$ for $t \in (-\delta,\delta)$ and show that it cannot be real-valued.

\begin{lemma}[The Non-real Lemma]\label{lmaNR}
For all $s_0 \in \mathbb{H}$ such that $\tet(s_0) \in \mathbb{R}$ there exists $\delta>0$ such for $t \in (-\delta,\delta)$ and $t\neq 0$ the values $\tet(s_0 + t) \not \in \mathbb{R}$.
\end{lemma}

\begin{proof}
Take our asymptotic solution to tetration,

\[
\beta(s) = \OmSum_{j=1}^\infty \frac{e^z}{e^{\frac{j-s}{\sqrt{1+s}}} + 1}\,\bullet z\\ 
\]

And note there are no lines $s_0 + t$ for $t \in \mathbb{R}$ in which this is real-valued, excepting when $s_0 \in \mathbb{R}$. Therefore, the limiting process,

\[
F_n(s) = \log^{\circ n}\beta(s+n)\\
\]

Cannot be real-valued on a line $s_0 + t$, unless $s_0 \in \mathbb{R}$; because this looks like $\beta(s_0 + t)$ for large $t$. Therefore the result.
\end{proof}

To justify the following theorem further we need only to add a small point. Let $y(z)$ be a holomorphic function such that $y(1) = 1$ and $y(0) = 0$. If we take the principal branch of the logarithm $\log : \mathbb{C}/\mathbb{R}_{x\le0} \to \mathbb{C}$, and we know $\log y(z)$ has a branch cut along $z \in \mathbb{R}_{x\le0}$, then necessarily $y(x) \in \mathbb{R}^+$ for $x\in(0,1)$. Which is nothing more than a mapping theorem on the $\log$ function.

\begin{lemma}\label{lmalog}
Suppose $y(z) : \mathbb{D}\to\mathbb{C}$ is a holomorphic function with $y(0)= 0$, and $y(1^-) = 1$. Suppose, using the principal branch of the logarithm,

\[
\log(y(z)) : \mathbb{D}/(-1,0]\to\mathbb{C}\\
\]

Then,

\[
y : [0,1] \to [0,1]\\
\]
\end{lemma}

\begin{proof}
Immediately, we must see that $y:(-1,0) \to -\mathbb{R}^+$; as this is where the branch-cut is located. Therefore $y:(0,1) \to (0,1)$; because once a holomorphic function is real-valued it will still be real-valued.
\end{proof}

With this, we state the following theorem, which will only require a quick justification.

\begin{theorem}[The Non-zero Theorem]\label{thmNZ}
The tetration function $\tet(s) \neq 0$ for $\Im(s) > 0$.
\end{theorem}

\begin{proof}
By The Non-real Lemma \ref{lmaNR}, we know for any point $\Im(s_0) > 0$ and $\tet(s_0) \in \mathbb{R}$, that $\tet(s_0 + t)$ is not real-valued for $t\in (-\delta,\delta)$. Choose an $s_0$ in which $\tet(s_0) = 1$.  The goal is to show that $\tet(s_0 - 1) = 2\pi i k$ for some $k \neq 0$.

We go by contradiction. Assume that $\tet(s_0 - 1) = 0$, then $\tet(s_0-2)$ is a singularity with a branch-point. This branching process can be done along the line $s_0 - 2 - t$ for $t \in \mathbb{R}^+$. Therefore $\tet(s+s_0-2)$ is holomorphic for $|s| < \delta$ and $s \not\in(-\delta,0]$. And further $y(s)=\tet(s+s_0-1) : \mathbb{D} \to \mathbb{C}$; where $y(1) = 1$ and $y(0) = 0$. 

And, we've used the principal branch of the logarithm to define $\log y(s) = \tet(s+s_0-2)$. Therefore $y(s) : [0,1] \to [0,1]$. Therefore the function $\tet(s+s_0-1)$ must be real-valued for $s \in (0,\delta)$, contradicting The Non-real Lemma \ref{lmaNR}.
\end{proof}

And with this we have constructed a tetration function holomorphic in the upper-half plane and real-valued analytic on the real-line. Therefore,

\begin{theorem}[The Tetration Theorem]\label{thmTET}
The function $\tet$ is holomorphic on $\mathbb{C}/(-\infty,-2]$, satisfies $\tet(0) = 1$ and,

\[
\tet(s+1) = e^{\tet(s)}\\
\]

And is given by the equation, for some $x_0 \in \mathbb{R}$,

\[
\tet(s) = \lim_{n\to\infty} \log^{\circ n} \beta(s+x_0+n)\\
\]

Where,

\[
\beta(s) = \OmSum_{j=1}^\infty \frac{e^z}{e^{\frac{j-s}{\sqrt{1+s}}}+1}\,\bullet z\\
\]
\end{theorem}

\section{Coding and Graphing}\label{secGRPH}
\setcounter{equation}{0}

We're going to spend a fair amount of time discussing manners of programming these various tetration functions; and attaching graphs based on the aforementioned coding method. The author has no efficient way of constructing the Taylor series of any of the tetrations. And so, the coding is largely done by bruteforce. We're going to write the first code dump in Matlab, but the theory extends to similar languages.

The second code dump will be more difficult, and is a hard-coded method for Pari-GP. This will produce much more accurate results, and will graph much more efficiently. Both code dumps are still in beta stages, but produce enough evidence to concur with the results of this paper\\

To begin, we call our function \texttt{beta(s,l,n)=} 

\[
{\displaystyle \beta_\lambda^n(s)=\OmSum_{j=1}^n \frac{e^z}{e^{\lambda(j-s)} + 1}\,\bullet z\Big{|}_{z=0}}
\]

Where $n$ can be thought of as the depth of iteration, and $l$ is the variable $\lambda$. All this is, is a for-loop. We can, somewhat, think of $\beta_\lambda^n \to \beta_\lambda$, as a for-loop where the depth of iteration is infinite. In that we'll write,

\begin{verbatim}
    function f = beta(s,l,n)
        f=0;
        for i = 0:n-1
            f = exp(f)./(1+exp(l*(n-i-s)));
        end
    end
\end{verbatim}

This is a very naive way of constructing the function $\beta_\lambda$, but it works well enough. This is the equivalent of an alternative method; which is to use the Taylor series of $g_\lambda(w) = \beta_\lambda(s)$ where $s \mapsto \log(w)/\lambda$. Where one would compute the Taylor coefficients of $g_\lambda(w)$ about $w=0$ and undo the substitution $s \mapsto \log(w)/\lambda$.

However, the naive way, is much simpler and much more generalizable. It's what we've used for all the graphs we've calculated here. It also more clearly displays the simple recursive nature of $\beta_\lambda$. Which, in many ways, is a self-referential for-loop iterated to infinity; $n\to\infty$. 

Now, when defining $\tau$ as code; the naive way is to fix an index $n$ in $\beta_\lambda$, and create an index $k$ in $\tau$; for its iteration. Mathematically speaking, this would be,

\[
\tau_{\lambda}^{n,k}(s) = \log\left(\beta_\lambda^n(s+1) - \tau_\lambda^{n,k-1}(s+1)\right) - \beta_\lambda^n(s)\\
\]

This can be written as the code:

\begin{verbatim}
    function f = tau(s,l,n,k)
        if k == 1
            f = log(beta(s+1,l,n)) - beta(s,l,n);
            return
        end

        f = log(beta(s+1,l,n) + tau(s+1,l,n,k-1)) - beta(s,n);
    end
\end{verbatim}

But, it equates, no less, to the function,

\[
\tau_\lambda^{n,k} = \log^{\circ k}\beta_\lambda^n(s+k) - \beta_\lambda^n(s)\\
\]

Which, as you may suspect; the function,

\[
\beta_\lambda^n(s+k) \to \infty\,\,\text{as}\,\,k\to\infty\\
\]

However, this tends to infinity too fast. If you try and graph this for large iterations, too many overflow errors happen in the circuit that everything over flows. Figure \ref{fig3-D} and \ref{fig3-D2} display clear anomalies where we've overflowed somewhere in the process. But also, they display the uniformity away from the singularities; where all is good. The points of divergence in the program do not equate to divergence in the math. The points of divergence in the program are just overflow errors. We get a short circuit.\\

\begin{figure}
    \centering
    \includegraphics[scale=0.4]{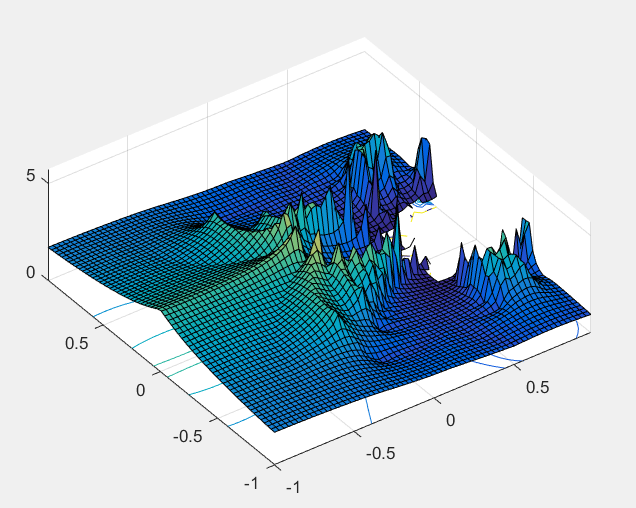}
    \caption{\texttt{y=beta(s,log(2),10) + tau(s,log(2),10,5)}; and we graphed \texttt{abs(y)} over the box $|\Re s|,|\Im s|\le1$.}
    \label{fig3-D}
\end{figure}

\begin{figure}
    \centering
    \includegraphics[scale=0.4]{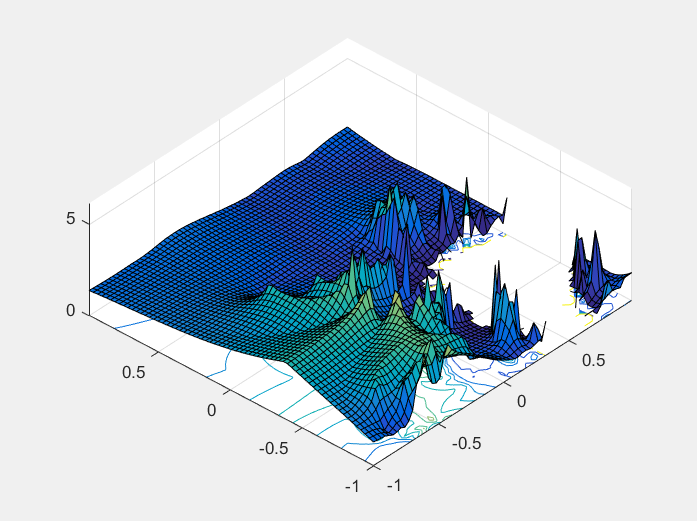}
    \caption{\texttt{y=beta(s,0.5+1i,10) + tau(s,0.5+1i,10,5)}; and we graphed \texttt{abs(y)} over the box $|\Re s|,|\Im s|\le1$.}
    \label{fig3-D2}
\end{figure}
The answer of which is to iterate while you iterate. We want to talk about the sequence of functions:

\begin{verbatim}
    tau_k(s,l,k) = tau(s,l,k,k)
\end{verbatim}

This can be written as the code,

\begin{verbatim}
    function f = tau_k(s,l,k)
        if k == 1
            f = log(beta(s+1,l,2)) - beta(s,l,2);
            return
        end

        f = log(beta(s+1,l,k) + tau(s+1,l,k,k-1)) - beta(s,l,k);
    end
\end{verbatim}

Of which we can see a larger more rapid area of convergence in Figure \ref{fig3-D3}. In the same breath, a more rapid divergence as well. We can continue to massage these things, by alternating method of constructing the iteration. These fairly accurately construct $F_\lambda(s)$; but it says nothing of the actual tetration we want.\\

\begin{figure}
    \centering
    \includegraphics[scale=0.5]{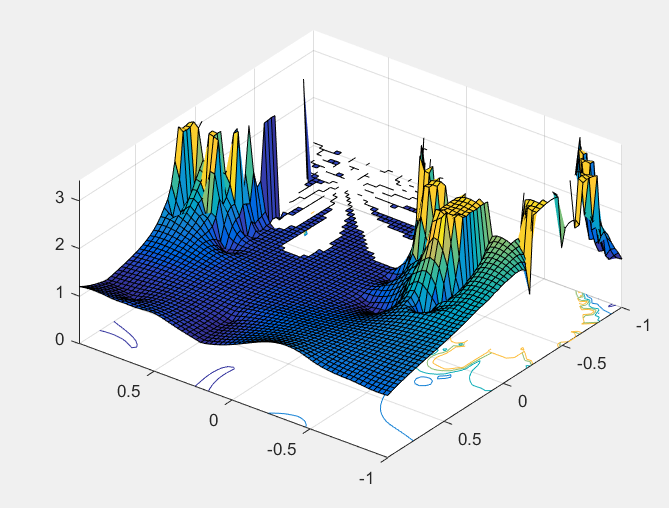}
    \caption{A faster converging view of tetration for $\lambda = 1/2 + i$. This is done with the function \texttt{beta(z,0.5+1i,10) + tau(z,0.5+1i,6)}, over the area $|\Re(s)|,|\Im(s)| \le 1$.}
    \label{fig3-D3}
\end{figure}

\begin{figure}
    \centering
    \includegraphics[scale=0.5]{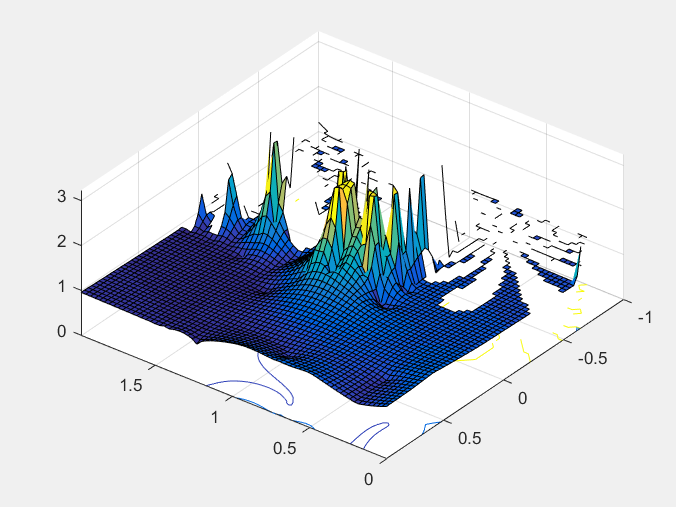}
    \caption{The same function, with an increased real argument. Again, we can observe a short circuit. This time it occurs because our function is beginning to grow super-exponentially.}
    \label{fig:my_label}
\end{figure}

For this, we'll need to introduce the $\beta$ function from The Pasted Together Theorem \ref{thmPST}. With that, call \texttt{beta2(s,n)=}

\[
\beta^n(s) = \OmSum_{j=1}^n \frac{e^z}{e^{\frac{j-s}{\sqrt{1+s}}} + 1}\,\bullet z\Big{|}_{z=0}\\
\]

This can be programmed as,

\begin{verbatim}
    function f = beta2(s,n)
        f=0;
        for i = 0:n-1
            f = exp(f)./(1+exp((n-i-s)./sqrt(1+s)));
        end
    end
\end{verbatim}

\begin{figure}
    \centering
    \includegraphics[scale=0.4]{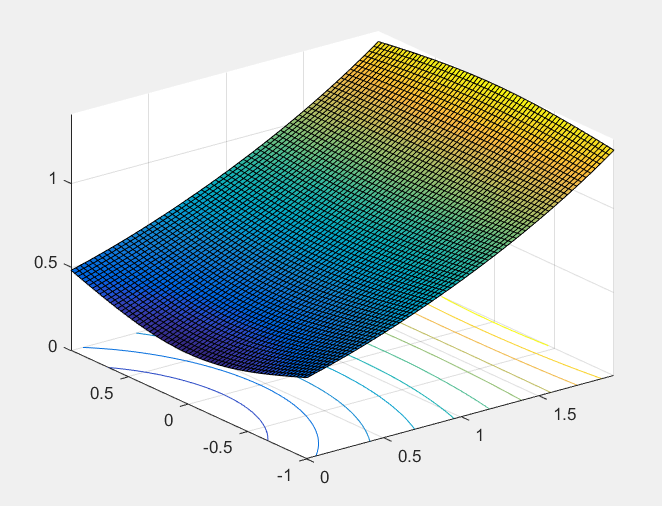}
    \caption{Here is a graph of the function $\beta(s)$--initialized as \texttt{beta2(s,100)} for $0 \le |\Re(s)| \le 2,\,|\Im(s)| \le 1$. We can see its very leveled growth to infinity pretty clearly.}
    \label{figBETAFN}
\end{figure}

Now coding a \texttt{tau2(s,n,k)} function is pretty straight forward.

\begin{verbatim}
    function f = tau2(s,n,k)
        if k == 1
            f = log(beta2(s+1,n)) - beta2(s,n);
            return
        end

        f = log(beta2(s+1,n) + tau2(s+1,n,k-1)) - beta2(s,n);
    end
\end{verbatim}

Trying to graph this will immediately produce overflow errors on the real-line, as we go out; as this grows far too large. But it looks leveled in the complex plane. Now, convergence on the real-line is rather trivial, and it is more important that this object converges in the complex plane. It looks somewhat like Figure \ref{figFINALTET}.

\begin{figure}
    \centering
    \includegraphics[scale=0.4]{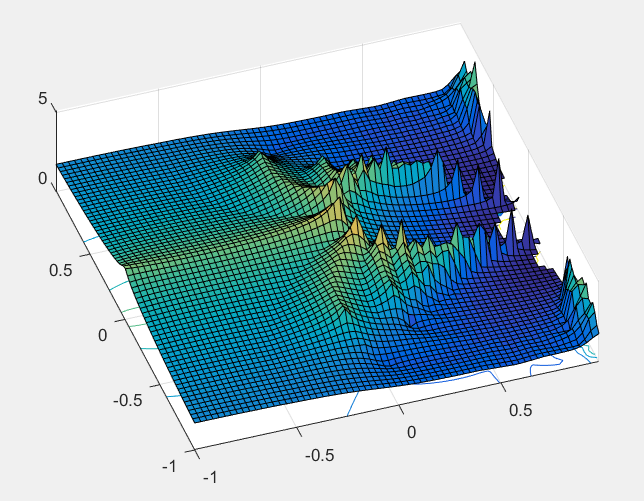}
    \caption{Here is \textit{almost} our tetration function over $|\Im(s)|,|\Re(s)| \le 1$; where we haven't yet shifted to the right domain. It begins to over flow as you increase the real argument, as the iteration takes too many large values of $\beta$; and we short-circuit. The ridge is the branch cut at $(-\infty,-2)$; and the edge of the graph dips to $0$ before shooting off to infinity--causing all the short-circuits.}
    \label{figFINALTET}
\end{figure}

Now, these views are of the iterative procedure. We can clean this code up a lot, by using a step function approach. With that, we attach,

\begin{verbatim}
    function f = TET(z)
        if (-1<real(z)<=0)
            f=beta2(z,8) + tau2(z,8,5);
            return
        end
    
        f = exp(TET(z-1));
    end
\end{verbatim}

And the graph in Figure \ref{figFINALTET2} looks much more accurate as to what our tetration looks like,\\

\begin{figure}
    \centering
    \includegraphics[scale=0.5]{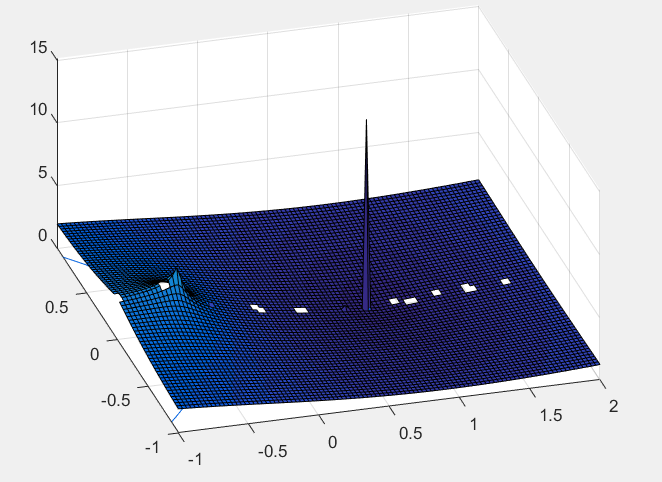}
    \caption{If we clean up the code of our previous iterations; and use the function \texttt{TET}; then this is what you get. This is precisely what our tetration should look like.}
    \label{figFINALTET2}
\end{figure}

Now onto the more difficult problem of programming in Pari-GP. The author will simply reference his GitHub repository for Tetration using Pari-GP. The functions are fairly similar, but slightly more difficult to use. The code dump can be found here \url{https://github.com/JmsNxn92/Recursive_Tetration_PARI}.

Instead, we'll begin by attaching a bunch of graphs which are constructed using these functions. These graphs are created using Graphing.gp in the GitHub Repository; which was not coded by me, but by the user \texttt{mike3} on The Tetration Forum \cite{Trp}. The following Figures \ref{fig:LOG_2} \ref{fig:LOG_POINT_1} \ref{fig:LOG_ONE_I} \ref{fig:ONE_PLUS_I} are all various graphs of $F_\lambda(s)$.\\

\begin{figure}
    \centering
    \includegraphics[scale=0.3]{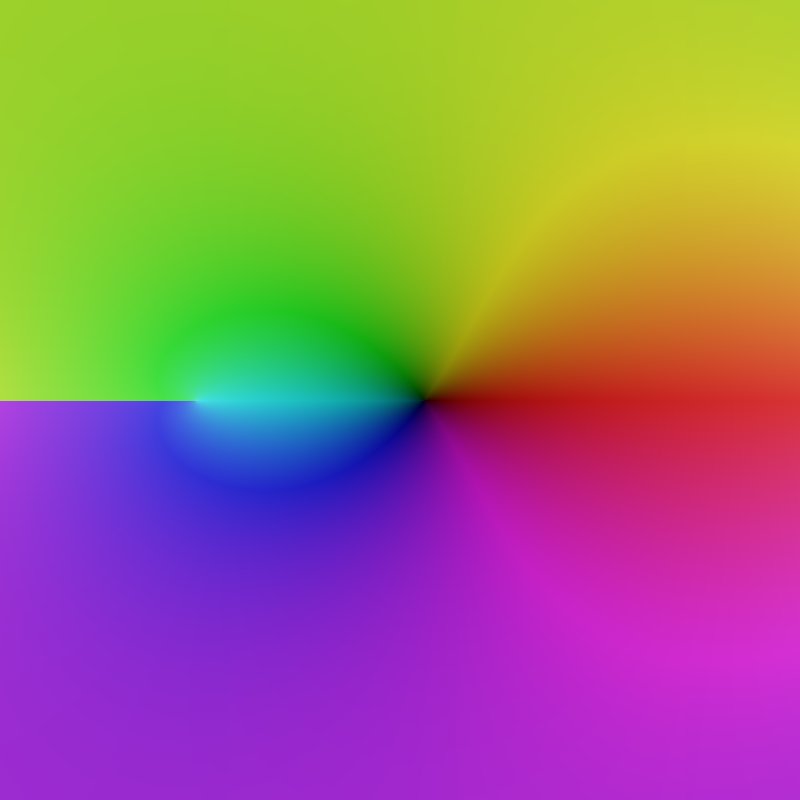}
    \caption{The function $F_{\log(2)}(s)$ for $-1 \le \Re(s) \le 3.5$ and $-2.5 \le \Im(s) \le 2.5$.}
    \label{fig:LOG_2}
\end{figure}

\begin{figure}
    \centering
    \includegraphics[scale=0.3]{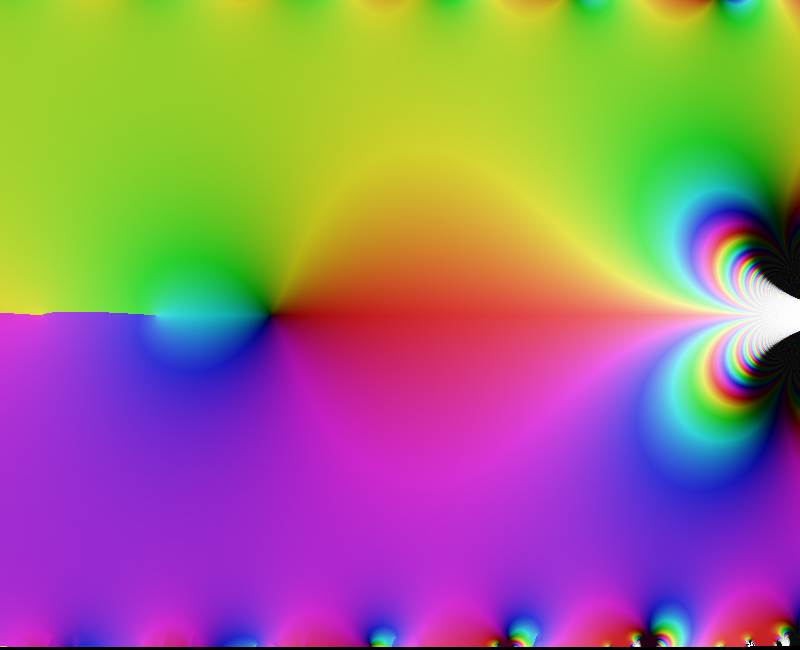}
    \caption{The function $F_{1+0.1i}(s)$ for $-1 \le \Re(s) \le 6$ and $3.2 \le \Im(s) \le 9$.}
    \label{fig:LOG_POINT_1}
\end{figure}

\begin{figure}
    \centering
    \includegraphics[scale=0.6]{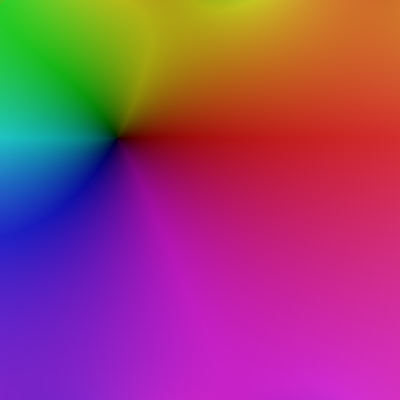}
    \caption{The function $F_{\log(2)-i}(s)$ for $0 \le \Re(s) \le 2$ and $-1 \le \Im(s) \le 1$.}
    \label{fig:LOG_ONE_I}
\end{figure}

\begin{figure}
    \centering
    \includegraphics[scale=0.6]{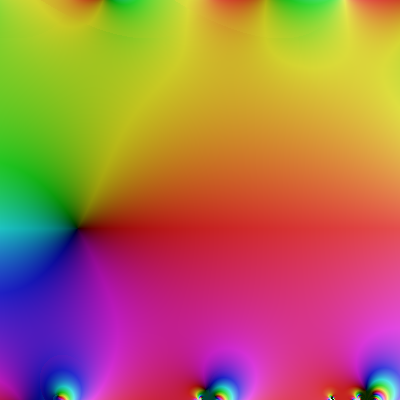}
    \caption{The function $F_{1+i}(s)$ for $0 \le \Re(s) \le 3$ and $-1.5 \le \Im(s) \le 1.5$.}
    \label{fig:ONE_PLUS_I}
\end{figure}

And attached further are various graphs for the function $\tet$.

\begin{figure}
    \centering
    \includegraphics[scale=0.4]{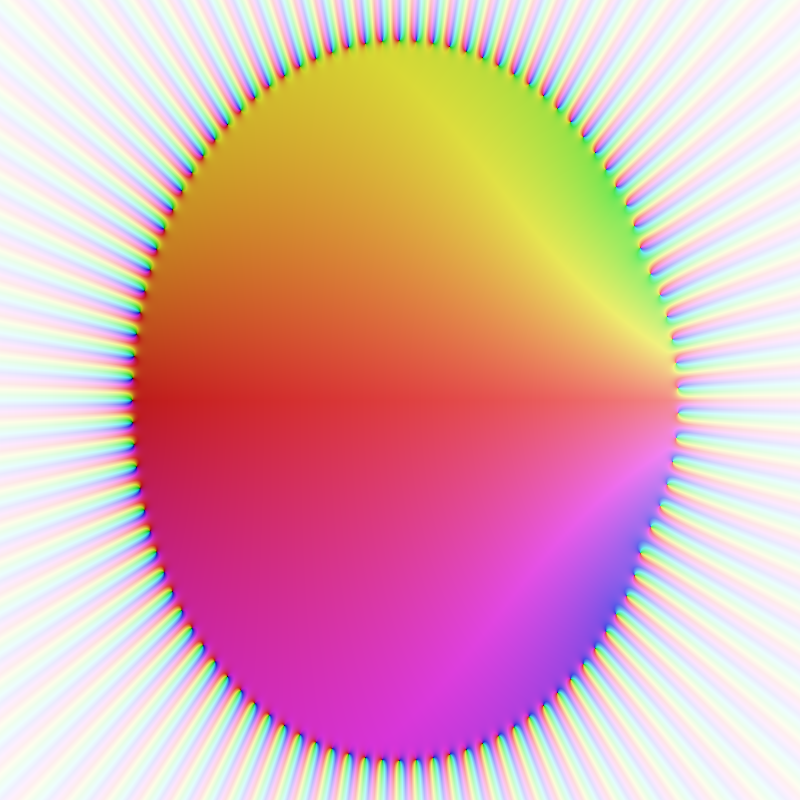}
    \caption{This is the Taylor series of $\tet(z)$ about $z = 1$ for 100 terms and 100 precision.}
    \label{fig:TAYLOR_TET}
\end{figure}

\begin{figure}
    \centering
    \includegraphics[scale = 0.4]{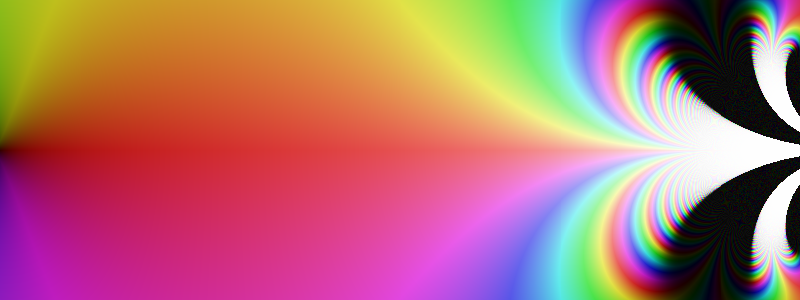}
    \caption{The function $\tet$ for $-1\le \Re(z) \le 4$ and $|\Im(z)| \le 0.8$.}
    \label{fig:TET_REAL}
\end{figure}

\section{Additional properties of $\tet$}\label{sec8}
\setcounter{equation}{0}

In this section we'll list some properties of $\tet$ which are extra to the general theory. These are nice things we can say about our solution. These are largely properties which are inherited from the family $\beta_\lambda$. Underlining how they are inherited is the important part. 

We start with our function $\beta(s)$ which we write as,

\[
\beta(s) = \OmSum_{j=1}^\infty \frac{e^z}{e^{\frac{j-s}{\sqrt{1+s}}}+1}\,\bullet z\\
\]

In which,

\[
\log\beta(s+1) - \beta(s) = \mathcal{O}(|s|^{-1/2})\\
\]

To intuitively see this, observe if we have a sequence of functions $h_j(s) = 1+\mathcal{O}(s^{-1/2})$, and we compose them to get the function,

\[
H(s) = \OmSum_{j=1}^\infty h_j(s)e^{z}\,\bullet z\\
\]

Then,

\[
\frac{H(s+1)}{e^{H(s)}} = h_1(s+1) e^{...}\\
\]

In our particular case, we know the exponent disappears, and $h_1(s+1) = 1+\mathcal{O}(|s|^{-1/2})$. Therefore the logarithm will behave as expected. This is a quick asymptotic, that the author hasn't used, and doesn't really plan on using; but it describes the shape of these things. Continuing this thread of discourse,

\[
\tet(s) - \beta(s+x_0) = \mathcal{O}(|s|^{-1/2})\,\,\text{as}\,\,|s|\to\infty\,\,\text{while}\,\,|\arg(s)| < \theta\\
\]

And so we can expect this with higher-order derivatives too, as the convergence is uniform for $|s| \to \infty$ while $|\arg(s)| < \theta$. Now differentiating $\beta$ is fairly easy; especially if we view this as a manner of computing Taylor coefficients of $g(w)$. We'll skip a few steps here, but the algebraic relationship as $\Re(s) \to \infty$,

\[
\beta'(s+1) \sim \frac{\beta'(s)e^{\beta(s)}}{e^{\frac{-s}{\sqrt{1+s}}}+1} +\frac{e^{\beta(s)}e^{\frac{-s}{ \sqrt{1+s}}}}{\left(e^{\frac{-s}{\sqrt{1+s}}}+1\right)^2} \left( \frac{d}{ds}\frac{-s}{\sqrt{1+s}} \right)
\]

Means that,

\[
\beta'(s) \sim \OmSum_{j=1}^\infty \frac{ze^{\beta(s-j)}}{e^{\frac{j-s}{\sqrt{1+s}}}+1} +\frac{e^{\beta(s-j)}e^{\frac{j-s}{ \sqrt{1+s}}}}{\left(e^{\frac{j-s}{\sqrt{1+s}}}+1\right)^2} \left( \frac{d}{ds}\frac{j-s}{\sqrt{1+s}} \right)\,\bullet z\\
\]

Which is eventually non-zero for large enough $s$. As such, we can expect that,

\[
\tet'(s) \neq 0\,\,\text{for}\,\,|s| > R\\
\]

And hereupon, since $\tet'(s) \neq 0$ we can derive that $\tet'(s-1) \neq 0$. Where,

\[
\tet'(s-1) = \frac{\tet'(s)}{\tet(s)} \neq 0
\]

Which implies that $\tet'(s) \neq 0$ everywhere $\tet(s) \neq 0$ which excludes the point $-1$ but at $-1$ we know that $\tet(-2)$ is a singularity, so this isn't a problem. Which is something really advantangeous to know, but to most studies of tetration, is rather apparent. This implies, yet again, that $\tet'(x) > 0$ for all $x \in (-2,\infty)$. This means that $\tet$ is a bijection of $(-2,\infty) \to \mathbb{R}$.

We can have a similar discussion with higher-order derivatives as well. It doesn't work out as nice, but is still worth while. Now, again we know that higher order derivatives of $\beta$ are also non-zero. As such, we get that,

\[
\tet^{(n)}(s) \neq 0\,\,\text{for}\,\,|s| > R_n\\
\]

This matters particularly for the real-line. Where it says that for large enough $X_n$, then for all $x>X_n$ we know that $\tet^{(n)}(x) > 0$. This tells us, eventually, each of our derivatives will be monotone. This is a slightly weaker criterion than all of its derivatives being monotone.\\

We'd like to take a quick moment to discuss the inverse function $\slog$ of $\tet$. We specifically refer to the $\slog$ function which takes $\mathbb{R} \to (-2,\infty)$ bijectively. This function will be analytic by the implicit function theorem.

Take $\mathcal{N}$ a neighborhood of zero in which $\slog$ is holomorphic. Then,

\[
\slog(e^z) = \slog(z) + 1\\
\]

This allows us to analytically continue $\slog$ to the set $\mathcal{S}$ in which,

\[
\mathcal{S} = \bigcup_{n=0}^\infty \exp^{\circ n}(\mathcal{N})\\
\]

Now, the orbits of the exponential map on an arbitrary neighborhood are dense in the complex plane. Which is the equivalent statement that the Julia set of $\exp$ is all of $\mathbb{C}$ (Again, we cite \cite{Mil,Lyu,Ree}). This amounts to $\overline{\mathcal{S}} = \mathbb{C}$. As such, we know that $\slog$ is holmorphic almost everywhere in $\mathbb{C}$; upto a measure zero set in $\mathbb{C}$.

This gives us a clear language that,

\[
\exp^{\circ s}(z) = \tet(s+\slog(z))\\
\]

Is holomorphic on a domain $\mathbb{P}$ in which $(s,z) \in \mathbb{P}$ and $\mathbb{C}^2/\mathbb{P}$ is a measure-zero set in $\mathbb{C}^2$. This function satisfies the functional equation,

\[
\exp^{\circ s}(\exp^{\circ s'}(z)) = \exp^{\circ s+s'}(z)\\
\]

For appropriately chosen $s$ and $s'$. This constructs what we'd think of as an appropriate fractional iteration of exponentiation; which satisfies the exponent law and takes real-values to real-values. Upon which the identity value $z \mapsto z$ is given at $s=0$ and $\tet(s)$ given at $z=1$.

If we fix $z$; this produces a holomorphic function in $s$ excepting branch cuts; and vice versa. Where the restriction $s,z\in\mathbb{R}^+$ implies $\exp^{\circ s}(z) \in \mathbb{R}^+$. This produces a family of functions ripe to construct pentation...

\section{In Conclusion}

To conclude this paper we broach the idea of doing this for more exotic functions. We ask if for other transcendental functions $h(z) : \mathbb{C} \to \mathbb{C}$, the asymptotic approach works to construct a super-function $H(z)$ such that $h(H(z)) = H(z+1)$; so $H$ satisfies the inverse Abel equation. Constructing an arbitrary function,

\[
\rho_\lambda(s) = \OmSum_{j=1}^\infty \frac{h(z)}{e^{\lambda(j-s)} + 1}\,\bullet z\\
\]

Which satisfies,

\[
\rho_\lambda(s+1) = \frac{h(\rho_\lambda(s))}{e^{-\lambda s} + 1}\\
\]

is not a difficult task, if the domains of $h$ are well behaved. But pulling back with iterates $h^{\circ -n}$ is a very careful procedure. Upon which, we were lucky with $e^z$ because $\log$ is a well behaved inverse. And despite the rapid growth of $e^z$, we were able to do this; where rapid growth is actually very beneficial. In essence, this method is more effective for rapid growing functions than it is for slowly growing functions.

Least of all, with these functions $\rho_\lambda$ we can describe asymptotically what $H$ should look like. Wherein, the equation,

\[
h^{-1}(\rho_\lambda(s+1)) - \rho_\lambda(s) = \mathcal{O}(e^{-\lambda s})\,\,\text{as}\,\,|s| \to \infty\,\,\Re(\lambda s) > 0\\
\]

Is certainly viable (so long as we have a decently well behaved function $h^{-1}$ at $\infty$). But without a decently behaved inverse $h^{-1}$, the most we'd be able to say is that $\rho_\lambda$ is a solution to the asymptotic inverse Abel equation--expressing the same thing but in a more implicit manner. The equation above being the frank way.

The author foresees no problem in utilizing this asymptotic method for $h(z) = b^z$ for $b > e^{1/e}$, where the iterates of $h$ are unbounded here. He imagines this would follow little differently than the case for $b=e$; subtracting minor details. Specifically, we would need a proof that $\rho_\lambda(s) \to \infty$ as $\Re(s) \to \infty$; in the same way we had $\beta_\lambda(s) \to \infty$. The complex plane $e^{\omega z}$ for $\omega \in \mathbb{C}$ is a different story though--it may be tractable, as long as its iterates are unbounded; though the complexity of the logarithms sounds like a serious headache.

The author also knows no way of understanding the dynamics of $\tet(s-n)$ for $\Im(s) > 0$. This equates to the repeated application of the logarithm; for varying branches of $\log$. The author is somewhat convinced this tetration $\tet \neq \text{tet}_K$, Kneser's tetration. Where in this regard, he expects the iterated $\log$'s on $\tet(s)$ may converge to varying fixed points, or diverge like the Julia set of the $\log$ map; and $\lim_{|s| \to \infty} \text{tet}_K(s) = L$ for $\pi/2 \le \arg(s) < \pi$. This is to say, Kneser's tetration is normal in the upper left half plane; the function $\tet$ is not.

With this, I conjecture that $\lim_{n\to\infty} \tet(s-n) \to L_{s},\,\infty$; where $L_s$ is a fixed point $e^L = L$. And $\infty$ means that $\tet(s) \in \mathcal{J}$ for $\mathcal{J}$ the Julia set of $\log$; upon which repeated applications don't converge. Infer, we interpret $\infty$ as non-normality, and $L_{s}$ as normality, and convergence towards a fixed point.

I, further, do not expect this solution to be Kneser's tetration because the behaviour as $\Im(s) = t \to \infty$ of $\tet$ should be $\infty$; as it should look like $\beta(it) + \mathcal{O}(e^{-\frac{it}{\sqrt{1+it}}})$, which $\beta(it)$ should tend to infinity (again the author isn't certain here, it just looks like it might work this way).\\

We thank the reader for their time, and their willingness to get to the bottom of this paper.

\section*{Appendix}\label{app}

We've attached here a proof of Theorem \ref{thmA}.

\begin{theorem}
Let $\{H_j(s,z)\}_{j=1}^\infty$ be a sequence of holomorphic functions such that $H_j(s,z) : \mathcal{S} \times \mathcal{G} \to \mathcal{G}$ where $\mathcal{S}$ and $\mathcal{G}$ are domains in $\mathbb{C}$. Suppose there exists some $A \in \mathcal{G}$, such for all compact sets $\mathcal{N}\subset\mathcal{G}$, the following sum converges,

\[
\sum_{j=1}^\infty ||H_j(s,z) - A||_{z \in \mathcal{N},s \in \mathcal{S}} = \sum_{j=1}^\infty \sup_{z \in \mathcal{N},s \in \mathcal{S}}|H_j(s,z) - A| < \infty
\]

Then the expression,

\[
H(s) = \lim_{n\to\infty}\OmSum_{j=1}^n H_j(s,z)\bullet z = \lim_{n\to\infty} H_1(s,H_2(s,...H_n(s,z)))\\
\]

Converges uniformly for $s \in \mathcal{S}$ and $z \in \mathcal{N}$ as $n\to\infty$ to $H$, a holomorphic function in $s\in\mathcal{S}$, constant in $z$.
\end{theorem}

\begin{proof}

The first thing we show is for all $\epsilon > 0$, there exists some $N$, such when $m \ge n  > N$,

\[
|\OmSum_{j=n}^{m} H_j(s,z)\bullet z - A| < \epsilon
\]

For $z$ in $\mathcal{N}\subset \mathcal{G}$ (where $A$ is in the open component of $\mathcal{N}$), and $s\in\mathcal{S}$. This then implies as we let $m\to\infty$, the tail of the infinite composition stays bounded. Forthwith, the infinite composition becomes a normal family, and proving convergence becomes simpler. We provide a quick proof of this inequality.\\

Set $||H_j (s, z)-A||_{\mathcal{S}, \mathcal{N}} = \rho_j$. Pick $\epsilon > 0$, and choose $N$ large enough so when $n > N$,

\[
\rho_n < \epsilon
\]

Denote: $\phi_{nm}(s, z) =\OmSum_{j=n}^m H_j (s, z) \bullet z = H_n(s, H_{n+1}(s, ...H_m(s, z)))$. We go by induction on the difference $m-n = k$. When $k=0$ then,

\[
||\phi_{nn}(s,z) - A||_{\mathcal{S},\mathcal{N}} = ||H_n(s,z)-A||_{\mathcal{S},\mathcal{N}}= \rho_n < \epsilon
\]

Assume the result holds for $m-n < k$, we show it holds for $m-n = k$. Observe,

\begin{eqnarray*}
||\phi_{nm}(s,z)-A||_{\mathcal{S},\mathcal{N}} &=& ||H_n(s,\phi_{(n+1)m}(s,z)) - A||_{\mathcal{S},\mathcal{N}}\\
&\le& ||H_n(s,z)-A||_{\mathcal{S},\mathcal{N}}\\
&=& \rho_n < \epsilon\\
\end{eqnarray*}

Which follows by the induction hypothesis because $\phi_{(n+1)m}(s,z) \subset \mathcal{N}$--it's in a neighborhood of $A$ which is in $\mathcal{N}$. That is $m-n-1 < k$.

The next step is to observe that $\OmSum_{j=1}^m H_j(s,z)$ is a normal family as $m\to\infty$, for $z \in \mathcal{N}$ and $s \in \mathcal{S}$. This follows because the tail of this composition is bounded. We can say $||\OmSum_{j=1}^m H_j(s,z)||_{\mathcal{S},\mathcal{N}} < M$ for all $m$.

Since $\phi_m(s,z) = \OmSum_{j=1}^m H_j(s,z)\bullet z$ are a normal family for all compact sets $\mathcal{N}\subset \mathcal{G}$; there is some constant $M \in \mathbb{R}^+$ and $L \in \mathbb{R}^+$ such,

\[
||\frac{d^k}{dz^k} \phi_m(s,z) ||_{\mathcal{S},\mathcal{N}} \le M \cdot k! \cdot L^k
\]\\

To see this, take $|z-A| < 2\delta$ and observe,

\[
\frac{d^k}{dz^k} \phi_m(s,z) = \frac{k!}{2\pi i}\int_{|\xi - A| = 2\delta} \frac{\phi_m(s,\xi)}{(\xi - z)^{k+1}}\,d\xi\\
\]

So that, taking the supremum norm across $|z-A| \le \delta$

\begin{eqnarray*}
||\frac{d^k}{dz^k} \phi_m(s,z)||_{\mathcal{S},|z-A| \le \delta} &\le& \frac{k!}{2\pi} \int_{|\xi-A| = 2\delta} \frac{||\phi_m(s,\xi)||_{\mathcal{S}}}{|\xi-z|_{|z-A| \le \delta}^{k+1}}\,d\xi\\
&\le& \frac{k!}{2\pi} \int_{|\xi-A| = 2\delta} \frac{M}{\delta^{k+1}}\,d\xi\\
&\le& \frac{2 M k!}{\delta^{k}}\\  
\end{eqnarray*}

Where we've used the bound $|\xi - z| \ge \delta$ when $|\xi - A| = 2\delta$ and $|z-A| \le \delta$. This bound can be derived regardless of $\mathcal{N}$ for varying $M$ and $L$.\\

Secondly, using Taylor's theorem,

\begin{eqnarray*}
\phi_{m+1}(s,z) - \phi_m(s,z) &=& \phi_m(s,H_{m+1}(s,z)) - \phi_m(s,z)\\
&=& \sum_{k=1}^\infty \frac{d^k}{dz^k} \phi_m(s,z) \frac{(H_{m+1}(s,z) - z)^k}{k!}\\
&=& (H_{m+1}(s,z) - z) \sum_{k=1}^\infty \frac{d^k}{dz^k} \phi_m(s,z) \frac{(H_{m+1}(s,z) - z)^{k-1}}{k!}\\
\end{eqnarray*}

So that, setting $z=A$,

\begin{eqnarray*}
||\phi_{m+1}(s,A) - \phi_m(s,A)||_{s \in\mathcal{S}} &\le& ||H_{m+1}(s,A) - A||_{s\in\mathcal{S}} \sum_{k=1}^\infty M L^k ||H_{m+1}(s,A) - A||^{k-1}\\
&\le& ||H_{m+1}(s,A) - A||_{\mathcal{S}} \frac{ML}{1-q}\\
\end{eqnarray*}

For $L||H_{m+1}(s,A) - A||_{\mathcal{S}} \le q <1 $, which is true for large enough $m>N$. Setting $C = \frac{ML}{1-q}$. Applying from here,

\[
||\phi_{m+1}(s,A) - \phi_m(s,A)||_{s \in \mathcal{S}} \le C ||H_{m+1}(s,A) - A||_{s \in \mathcal{S}}\\
\]

This is a convergent series per our assumption. Choose $N$ large enough, so that when $m,n>N$,

\[
\sum_{j=n}^{m-1}||H_{j+1}(s,A) - A||_{s \in \mathcal{S}} < \frac{\epsilon}{C}\\
\]

Then,

\begin{eqnarray*}
||\phi_{m}(s,A) - \phi_n(s,A)||_{s \in \mathcal{S}} &\le& \sum_{j=n}^{m-1} ||\phi_{j+1}(s,A) - \phi_j(s,A)||_{s \in \mathcal{S}}\\
&\le& C\sum_{j=n}^{m-1}||H_{j+1}(s,A) - A||_{s \in \mathcal{S}}\\
&<& \epsilon
\end{eqnarray*}

So we can see $\phi_m(s)$ must be uniformly convergent for $s \in \mathcal{S}$, and therefore defines a holomorphic function $H(s)$ as $m\to\infty$.

This tells us,

\[
H(s) = \OmSum_{j=1}^\infty H_j(s,z)\bullet z \Big{|}_{z=A}\\
\]

Converges and is holomorphic. To show this function equals,

\[
\OmSum_{j=1}^\infty H_j(s,z)\bullet z
\]

For all $z \in \mathcal{G}$; simply notice that,

\[
\OmSum_{j=m}^\infty H_j(s,z)\bullet z
\]

Is arbitrarily close to $A$ as we let $m$ grow (which was shown at the beginning of this proof). Then,

\begin{eqnarray*}
\OmSum_{j=1}^\infty H_j(s,z)\bullet z &=& \OmSum_{j=1}^{m-1} H_j(s,z)\bullet \OmSum_{j=m}^\infty H_j(s,z)\bullet z\\
&=& \lim_{m\to\infty} \OmSum_{j=1}^{m-1} H_j(s,z)\bullet \lim_{m\to\infty} \OmSum_{j=m}^\infty H_j(s,z)\bullet z\\
&=& \OmSum_{j=1}^\infty H_j(s,z)\bullet z\Big{|}_{z=A}\\
\end{eqnarray*}

This concludes our proof.

\end{proof}

\section*{Acknowledgements}

I'd like to thank the community at The Tetration Forum \cite{Trp}. Without this community, I do not think this paper would've come into fruition the manner that it has. In this regard, I would not have been able to spot the error in my initial construction using a function $\phi$ which satisfied a different criterion than the family $\beta_\lambda$. To this, I owe many thanks to Sheldon Levenstein particularly.

I also owe a great deal of thanks to Tom Marcel Raes--who was very persistent in the belief that there may be some function in which the original $\phi$ method would work--which was pushed to its extreme in \cite{Nix2}, but only for real-values. I had significant doubt after the failure of holomorphy with $\phi$, that no function could properly work--produce a holomorphic tetration. But upon the many suggested functions he had, an idea formed--make it solve the tetration equation at $\infty$! Which led me to consider the family of functions $\beta_\lambda$--which are asymptotically tetration.

I'd also like to thank the user MphLee, who by happenstance asked me a bunch of questions, which led myself to consider the logistic map $\dfrac{1}{e^{-\lambda s} + 1}$ as our multiplier rather than $e^s$ (as it was with the $\phi$ method). Whereupon, this was the function which worked most universally for the construction we needed to answer his questions--or at least attempt to do so. I'd also like to thank him for all of the commutative diagrams he made--which encouraged me to use my own. He aided me greatly with many of his funny drawings.

I'd also like to thank Henryk Trappman, for of course curating and creating the forum; and for the disappointment that there were no pictures in the first paper \cite{Nix2}, encouraging me to include graphics. And I'd like to thank the many members of the forum who have helped me understand these problems throughout the course of my mathematical development. Where, there exists posts from many years gone by when I was no more than a teenager.

Regards, James

\end{document}